\documentclass[onefignum,onetabnum]{siamonline190516}

\usepackage{lipsum}
\usepackage{amsfonts,mathrsfs}
\usepackage{graphicx}
\usepackage{epstopdf}
\usepackage{algorithmic}
\usepackage{multirow,booktabs,cases,bm}
\ifpdf
\DeclareGraphicsExtensions{.eps,.pdf,.png,.jpg}
\else
\DeclareGraphicsExtensions{.eps}
\fi

\usepackage{enumitem}
\setlist[enumerate]{leftmargin=.5in}
\setlist[itemize]{leftmargin=.5in}

\newsiamremark{remark}{Remark}
\newsiamremark{assumption}{Assumption}
\newsiamremark{hypothesis}{Hypothesis}
\crefname{hypothesis}{Hypothesis}{Hypotheses}
\newsiamthm{claim}{Claim}

\newcommand{\lno}{\left\|}
\newcommand{\rno}{\right\|}
\newcommand{\dom}[1]{\mathrm{\bf dom}\,{(#1)}} 
\newcommand{\intset}[1]{\mathrm{\bf int}\,{(#1)}} 
\newcommand{\graph}[1]{\mathrm{\bf graph}\,{(#1)}} 
\newcommand{\crit}[1]{\mathrm{\bf crit}\,{(#1)}} 
\newcommand{\levelset}{\mathrm{\bf Lev}} 
\newcommand{\prox}{\mathrm{\bf Prox}} 
\newcommand{\dist}{\mathrm{\bf dist}} 
\newcommand{\shrink}{\mathrm{\bf shrink}}
\newcommand{\Proj}{\mathrm{\bf Proj}}
\newcommand{\SVT}{\mathrm{\bf SVT}}
\newcommand{\RR}{\mathbb{R}}
\newcommand{\NN}{\mathbb{N}}
\def\B{\mathscr{B}}
\def\D{\mathbf{D}}
\newcommand{\ww}{\bm{w}}
\newcommand{\uu}{\bm{u}}
\newcommand{\vv}{\bm{v}}
\newcommand{\zz}{\bm{z}}

\headers{A Unified Bregman Alternating Minimization Algorithm}{H. He and Z. Zhang}

\title{A Unified Bregman Alternating Minimization Algorithm for Generalized DC Programming with Application to Imaging Data\thanks{Submitted to the editors XX 202X.
		\funding{This work was partially supported by National Natural Science Foundation of China (No. 11771113) and Ningbo Natural Science Foundation (Project ID: 2023J014).}}}

\author{Hongjin He\thanks{School of Mathematics and Statistics, Ningbo University, Ningbo 315211, China.
		(\email{hehongjin@nbu.edu.cn}).}
	\and Zhiyuan Zhang\thanks{School of Mathematical Sciences, Xiamen University, Xiamen 360015, China.
		(\email{zhang510zg@gmail.com}).}
}

\usepackage{amsopn}

\allowdisplaybreaks
\begin{document}
\graphicspath{{./PICs/}}

\maketitle

\begin{abstract}
  In this paper, we consider a class of nonconvex (not necessarily differentiable) optimization problems called generalized DC (Difference-of-Convex functions) programming, which is minimizing the sum of two separable DC parts and one two-block-variable coupling function. To circumvent the nonconvexity and nonseparability of the problem under consideration, we accordingly introduce a Unified Bregman Alternating Minimization Algorithm (UBAMA) by maximally exploiting the favorable DC structure of the objective. Specifically, we first follow the spirit of alternating minimization to update each block variable in a sequential order, which can efficiently tackle the nonseparablitity caused by the coupling function. Then, we employ the Fenchel-Young inequality to approximate the second DC components (i.e., concave parts) so that each subproblem reduces to a convex optimization problem, thereby alleviating the computational burden of the nonconvex DC parts. Moreover, each subproblem absorbs a Bregman proximal regularization term, which is usually beneficial for inducing closed-form solutions of subproblems for many cases via choosing appropriate Bregman kernel functions. It is remarkable that our algorithm not only provides an algorithmic framework to understand the iterative schemes of some novel existing algorithms, but also enjoys implementable schemes with easier subproblems than some state-of-the-art first-order algorithms developed for generic nonconvex and nonsmooth optimization problems. Theoretically, we prove that the sequence generated by our algorithm globally converges to a critical point under the Kurdyka-{\L}ojasiewicz (K{\L}) condition. Besides, we estimate the local convergence rates of our algorithm when we further know the prior information of the K{\L} exponent. A series of numerical experiments on imaging data demonstrate the reliability and efficiency of the proposed algorithmic framework.
\end{abstract}



\section{Introduction}\label{Sec:Intro}
With the rapid developments of artificial intelligence, nonconvex (differentiable/nondifferentiable) optimization problems frequently appear and play important roles in the areas of data sciences, imaging processing, computer visions, and machine learning, e.g., see some surveys \cite{CD20,CLC19,CP21,DDGGGKS22,RHLNSH20,Val21} and the most recent monograph \cite{CP22} on nonconvex topics. As the backbone of nonconvex programming and global optimization, in recent years, we have witnessed the fruitful developments of DC (Difference-of-Convex functions) programming and DCA (DC Algorithms) from theory to their applications. Here, we refer the reader to a recent survey \cite{LTPD18} and reference therein for thirty years of developments along this direction. As stated by Le Thi and Pham Dinh in \cite{LTPD18}: ``{\it Despite the bright successes obtained by researchers and practitioners in the use of DC programming and DCA for modeling and solving nonconvex and global optimization problems, their works have not exploited the full power and creative freedom offered by these tools: their proposed DCAs, although more efficient than existing methods, but can be further improved to better handle large-scale problems}'', we are motivated to consider a class of structured DC programming and to design a structure-exploiting algorithm to solve the underlying problem.
Specifically, in this paper, we are interested in a class of generalized DC programming, which reads as
\begin{equation}\label{eq:problem}
\min_{x\in \RR^n, y\in \RR^m} \; \left\{\;\Phi(x,y) :=\underbrace{f_1(x)-g_1(x)}_{\theta_1(x)}+\underbrace{f_2(y)-g_2(y)}_{\theta_2(y)}+\underbrace{h^+(x,y)-h^-(x,y)}_{h(x,y)}\;\right\},
\end{equation}
where $f_i(\cdot)$ and $g_i(\cdot)$ are convex (not necessarily smooth) functions for $i=1,2$, and $h(x,y)$ is a continuously differentiable but coupling function, which is also assumed to be decomposable into two parts $h^+(x,y)$ and $h^-(x,y)$ similar yet slightly different to the case studied in \cite{LPLV15,PHLH22}. Unlike the standard DC programming studied in the literature where the objective \eqref{eq:problem} possesses only one block variable (e.g., see \cite{BB19,GZLHY13,GTT18,PDLT97,LZOX15,LZ19,LZS19,MLH17,MM08,YLHX15,ZBN20}, to name just a few), our model \eqref{eq:problem} under consideration can be regarded as the complementary or generalized case to the typical DC programming. Here, it is worth pointing out that we only need to assume both components $f_1(\cdot)$ and $f_2(\cdot)$ being proper lower semicontinuous functions for convergence analysis, which makes our approach applicable to a wider range of nonconvex optimization models. Besides, we restrict the later discussion to the case of \eqref{eq:problem} with vector variables, all of our results are also available to the case with matrix variables (see Section \ref{Sec:Exp}). 

When applying the traditional DCA-type algorithms (e.g., \cite{AV20,CHZ22,deOT19,KW19,WCP18}) tailored for standard DC programming, the main computational challenge is that the appearance of coupling part $h(x,y)$ makes these algorithms difficult to be implemented in many cases. Accordingly, a natural and efficient way to handle block coupling optimization problems is the so-named alternating minimization (a.k.a. block coordinate descent or Gauss-Seidel) algorithm (e.g., see \cite{Beck17,BT89,Wright-2015}), which iteratively and alternately keeps one of block variables fixed and optimizes over the other. More concretely, for given initial point $(x^0,y^0) \in \RR^n \times \RR^m$, the iterative scheme of the alternating minimization algorithm reads as
\begin{equation}\label{eq:AM}
\left\{
\begin{aligned}
x^{k+1} & \in  \arg\min_{x \in \RR^n} \; \Phi(x,y^k),\\
y^{k+1} & \in  \arg\min_{y \in \RR^m} \; \Phi(x^{k+1},y).
\end{aligned}\right.
\end{equation}
Unfortunately, it has been documented in \cite{Powell-1973} that \eqref{eq:AM} may cycle indefinitely without converging if the minimum in each step cannot be uniquely attained. Accordingly, in \cite{BT13,Ber16,GS99,LT93,Tse01}, the authors studied some convergence behaviors of \eqref{eq:AM} under various assumptions, such as pseudoconvexity, quasiconvexity, convexity and continuity, on the objective $\Phi(x,y)$. In particular, some papers (e.g., \cite{Aus92,GS00}) are contributed to removing the strict convexity assumption by imposing proximal regularization terms into the subproblems of \eqref{eq:AM}, i.e.,
\begin{equation}\label{eq:pAM}
\left\{\begin{aligned}
x^{k+1} & =\arg\min_{x \in \RR^n} \; \left\{\Phi(x,y^k) + \frac{c_k}{2} \lno x - x^k \rno^2\right\},\\
y^{k+1} & = \arg\min_{y \in \RR^m} \; \left\{\Phi(x^{k+1},y) + \frac{d_k}{2} \lno y - y^k \rno^2\right\},
\end{aligned}\right.
\end{equation}
where $c_k$ and $d_k$ are positive numbers. Clearly, both subproblems of \eqref{eq:pAM} are strongly convex as long as the objective $\Phi(\cdot,\cdot)$ is convex of one argument while the other is fixed, which can easily guarantee each subproblem to have a unique solution.
Furthermore, some novel works \cite{ABRS10,ABS13} are devoted to the convergence of \eqref{eq:pAM} for handling generic nonconvex cases under the well-known Kurdyka-{\L}ojasiewicz (K{\L}) property (see \cite{Kurdyka-1998,Loj63} and also \cite{BDL07}).
However, the iterative scheme \eqref{eq:pAM} probably suffers from difficult subproblems in many cases, for example when applying to nonnegative matrix/tensor factorization problems (see \cite{AHGP21,BST14,XY13}), even though the first two separable functions $\theta_1(x)$ and $\theta_2(y)$ are void, and hence is potentially slow with a comparatively high computational complexity in practice. Therefore, to alleviate the computational cost caused by the coupling function, Bolte et al. \cite{BST14} judiciously introduced a Proximal Alternating Linearized Minimization (PALM) algorithm, which considers an approximation of \eqref{eq:pAM} via the standard proximal linearization of each subproblem, thereby maximally making its subproblems enough easy with closed-form solutions in some cases. Specifically, the PALM algorithm updates $x$ and $y$ alternately via
\begin{equation}\label{eq:PALM}
\left\{\begin{aligned}
x^{k+1} & = \arg\min_{x \in \RR^n} \; \left\{\theta_1(x) + \left\langle \nabla_x h(x^k,y^k), x-x^k \right\rangle + \frac{c_k}{2} \lno x - x^k \rno^2\right\},\\
y^{k+1} & = \arg\min_{y \in \RR^m} \;\left\{\theta_2(y) + \left\langle \nabla_y h(x^{k+1},y^k), y-y^k \right\rangle + \frac{d_k}{2} \lno y - y^k \rno^2\right\},
\end{aligned}\right.
\end{equation}
where $\nabla_x h(\cdot,y)$ and $\nabla_y h(x,\cdot)$ represent the block-partial gradients of $h(x,y)$, respectively. Moreover, some convergence results of \eqref{eq:PALM} were established under the K{\L} property and the Lipschitz continuity of $\nabla h(x,y)$. Alternatively, when both $\theta_1(\cdot)$ and $\theta_2(\cdot)$ are assumed to be differentiable, the authors \cite{LTP22,NT19} considered locally quadratic approximations of $\theta_1(\cdot)$ and $\theta_2(\cdot)$, and meanwhile exploited potentially explicit forms of the proximal operators with respect to the partial functions $h(\cdot,y^k)$ and $h(x^{k+1},\cdot)$, respectively. In the recent literature, a series of novel variants of \eqref{eq:pAM} and \eqref{eq:PALM} are developed for generic nonconvex and nonsmooth optimization problems, e.g., see \cite{AHGP21,BCN19,GCH20,HPG23,HWRL17,OCBP14,PS16,XY13,XY17}, to mention just a few.  

Actually, it is not difficult to observe that the efficiency of \eqref{eq:AM}, \eqref{eq:pAM}, and \eqref{eq:PALM} heavily depends on the complexity of their subproblems. With the above literature review, we can see that, when applying the aforementioned alternating minimization algorithms and their variants to our model \eqref{eq:problem}, these algorithms ignore the DC structure of the objective function, thereby still possibly suffering from difficult subproblems. For instance, when the nonconvex function $\theta_1(\cdot)$ (or $\theta_2(\cdot)$) is specified as the nonsmooth capped norms (see \cite{SXY13} or our experiments in Section \ref{Sec:Exp}), the algorithms \eqref{eq:AM}, \eqref{eq:pAM}, and \eqref{eq:PALM} must solve a DC optimization subproblem via an inner loop, which is often an expensive procedure for solving large-scale problems. Therefore, how to design an algorithm that is able to make full use of the favorable DC structure of model \eqref{eq:problem} is the main motivation of this paper. 

In this paper, we aim to introduce a Unified Bregman Alternating Minimization Algorithm (UBAMA) to deal with the generalized DC programming \eqref{eq:problem}. Our approach exploits the block-variable structure to update variables in a sequential order so that it can easily circumvent the nonseparability appeared in the objective function. 
Then, instead of solving DC optimization subproblems directly, we consider approximations to the second DC components and the coupling part by employing the Fenchel-Young inequality and Bregman proximal regularization, respectively. Unlike many traditional DC algorithms for one-block DC programming, one remarkable advantage of our approach is that it does not require the second DC parts (i.e., $g_1(\cdot)$ and $g_2(\cdot)$) being differentiable. Another remarkable advantage is that the embedded Bregman proximal regularization terms make our approach easily implementable and flexible for many cases. 
In other words, our approach provides an algorithmic framework to understand the iterative schemes of some state-of-the-art first-order algorithms proposed in \cite{ABRS10,BST14,NT19,PHLH22,SSC03,TFT22} (see remarks in Section \ref{Sec:Alg}) for nonsmooth and nonconvex optimization, so that these algorithms can be concatenated together. Besides, our approach allows us to derive some customized algorithms for some specific applications in the sense that the resulting algorithms often enjoy easy subproblems with closed-form solutions (see examples Section \ref{Sec:Exp}). 
Theoretically, we prove that the sequence generated by our algorithm is globally convergent to a critical point of \eqref{eq:problem} under the K{\L} condition. Besides, we give the estimation of the local convergence rates of our algorithm under extra prior information of K{\L} exponent. Finally, we modify some imaging optimization models to fall into the form of \eqref{eq:problem} for the purpose of verifying the reliability of our proposed algorithm. A series of numerical experiments on imaging data further demonstrate that our approach is promising and reliable for solving generalized DC programming \eqref{eq:problem}.

The remainder of this paper is organized as follows. In Section \ref{Sec:Pre}, we introduce some preliminaries and recall some basic properties for nonconvex optimization. In Section \ref{Sec:Alg}, we present our algorithm and give some remarks to state that some classical iterative schemes for nonconvex optimization problems fall into the framework of our approach. In Section \ref{Sec:Converg}, we establish some convergence properties of the proposed algorithm. In Section \ref{Sec:Exp}, some numerical experiments on imaging data are conducted to support the idea of this paper. Finally, we close this paper with drawing some concluding remarks in Section \ref{Sec:Con}.

\section{Preliminaries} \label{Sec:Pre}
In this section, we summarize some notions, definitions, and basic properties of K{\L} inequality from (non) convex and nonsmooth analysis that will be used throughout this paper (e.g., see \cite{Roc70,RW98}).

Let $\RR^n$ be an $n$-dimensional Euclidean space equipped with the standard inner product of vectors defined by $\langle \cdot,\cdot \rangle$. For a given $x\in \RR^n$ and $1\leq p <\infty$, we denote $\lno x\rno_p=\left(\sum_{i=1}^{n}|x_i|^p\right)^{\frac{1}{p}}$ as the $\ell_p$ norm of $x$, where $x_i$ is the $i$-th component of vector $x$. In particular, we denote the $\ell_2$-norm (i.e., Euclidean norm) of $x\in\RR^n$ by $\lno x\rno\equiv \sqrt{ \langle x,\;x \rangle}$. Moreover, letting $M$ be a positive definite matrix (denoted by $M\succ0$), we define the $M$-norm of $x\in\RR^n$ by $\lno x\rno_M = \sqrt{\langle x, M x\rangle}$. Given a matrix $X\in \RR^{m\times n}$, we denote the nuclear norm of matrix $X$ by $\lno X\rno_*:= \sum_{i=1}^{\min\{m,n\}} \sigma_i(X)$, where $\sigma_i(X)$  is the $i$-th largest singular value of $X$. Besides, we denote $\|X\|_F$ as the standard Frobenius norm of matrix $X$, and $\|X\|_1$ is an extension of the $\ell_1$ norm of vectors for matrices, which represents the sum of absolute values of all entries.

Below, we first recall some basic definitions for the coming analysis.
\begin{definition}\label{def:lsc}
	Let $f(\cdot):\RR^n \to [-\infty,+\infty]$ be an extended-real-valued function, and denote the domain and level set of $f(\cdot)$ by 
	$$\dom{f} := \left\{x\in \RR^n\;|\;f(x) <\infty\right\}\quad \text{and}\quad \levelset_f(\alpha)=\{x\in \RR^n\;|\; f(x) \le \alpha\},$$ respectively. Then, we say that the function $f(\cdot)$ is
	\begin{enumerate}			
		\item[\rm (i)] proper if $f(x)  >-\infty$ for all $x \in \RR^n$ and $\dom{f}\neq \emptyset$;
		\item[\rm (ii)]  lower semicontinuous if and only if the level set $\levelset_f(\alpha)$ is closed for every $\alpha\in\RR$;
		\item[\rm (iii)] lower level-bounded if its level set $\levelset_f(\alpha)$ is bounded (possibly empty) for every $\alpha \in \RR$;
		\item[\rm (iv)] convex if $f\left(tx+(1-t)y\right)\leq t f(x) +(1-t)f(y)$ for all $x,y\in \dom{f} $ and $t\in[0,1]$;
		\item[\rm (v)] $\varrho$-strongly convex for a given $\varrho>0$ if $\dom{f}$ is convex and the following inequality holds for any $x,y\in \dom{f}$ and $t\in[0,1]$:
		$$f(tx + (1-t)y)\leq t f(x) + (1-t)f(y) - \frac{\varrho}{2}t(1-t)\|x-y\|^2.$$
	\end{enumerate}
\end{definition}

\begin{definition}\label{def:subdiff}
	Let $f(\cdot): \RR^n \to (-\infty,+\infty]$ be a proper and lower semicontinuous function.
	\begin{enumerate}
		\item[\rm (i)] For each $x \in \dom{f}$, the Fr\'{e}chet subdifferential $\widehat{\partial}f(x)$ of $f(\cdot)$ at $x$ is defined by
		\begin{equation*}
		\widehat{\partial}f(x)  :=  \left\{ \xi\in\RR^n \;\Big{|}\; \liminf_{\substack{y\ne x\\ y\to x}} \frac{f(y)-f(x) -\langle \xi,y-x \rangle}{\|y-x\|} \ge 0\right\}.
		\end{equation*}
		In particular, when $x \notin \dom{f}$, we set $\widehat{\partial}f(x)  = \emptyset$.
		\item[\rm (ii)] The limiting subdifferential $\partial f(x)$ of $f(\cdot)$ at $x\in \dom{f}$ is defined by
		\begin{equation*}
		\partial f(x)  := \left\{\xi \in \RR^n\;\Big{|}\; \begin{array}{l}
		\exists \left(x^k,f(x^k)\right) \to \left(x,f(x)\right), \xi^k \in \widehat{\partial}f(x^k) \\ \text{such that  }\; \xi^k\to \xi\; \text{  as } \;k\to +\infty 
		\end{array} \right\}.
		\end{equation*}
	\end{enumerate}
\end{definition}

It follows from Definition~\ref{def:subdiff} that $\widehat{\partial}f(x)  \subset \partial f(x) $ for each $x \in \RR^n$, where the first set $\widehat{\partial}f(x) $ is convex and closed while the second one $ \partial f(x)$  is closed (e.g., see \cite[Theorem 8.6]{RW98}). 
Let $(x^k,\xi^k) \in \graph{\partial f}:= \left\{(x,\xi)\in \RR^n \times \RR^n\;|\;\xi \in \partial f(x) \right\}$ be a sequence. If $(x^k,\xi^k)$ converges to $(x,\xi)$ and $f(x^k)$ converges to $f(x) $, then $(x,\xi) \in \graph{\partial f}$. 
When $f(\cdot)$ is continuously differentiable, the limiting subdifferential reduces to the gradient of $f(\cdot)$, denoted by $\nabla f$ (e.g., see \cite[Exercise 8.8(b)]{RW98}). Moreover, when $f(\cdot)$ is convex, the limiting subdifferential reduces to the classical subdifferential in convex analysis (see \cite[Proposition 8.12]{RW98}).
A necessary but not sufficient condition for $x^* \in \RR^n$ to be a local minimizer of $f(\cdot)$ is that $x^*$ is a (limiting-) critical point, i.e., $0 \in \partial f(x^*).$ Throughout this paper, the set of critical points of $f(\cdot)$ is denoted by $\crit{f}:= \left\{x\;|\;0\in \partial f(x) \right\}$.

As shown in the literature (e.g., see \cite[Chapter 5]{Beck17}), the following first-order characterizations of strong convexity are frequently used for analysis.
\begin{lemma}\label{lem:strong}
	Let $f(\cdot): \RR^n\to (-\infty,\infty]$ be a proper closed and convex function. Then, for a given $\varrho>0$, the following three claims are equivalent:
	\begin{itemize}
		\item[\rm (i)] $f(\cdot)$ is $\varrho$-strongly convex.
		\item[\rm (ii)] $f(\cdot)-\frac{\varrho}{2}\|\cdot\|^2$ is convex. 
		\item[\rm (iii)] $f(y)\geq f(x)+\langle \xi, y-x\rangle + \frac{\varrho}{2}\|y-x\|^2$ for any $x\in\dom{\partial f}$, $y\in\dom{f}$, and $\xi\in\partial f(x)$.
		\item[\rm (iv)] $\langle \xi-\eta, x-y\rangle \geq \varrho \|x-y\|^2$ for any $x,y\in\dom{\partial f}$ and $\xi \in\partial f(x)$, $\eta\in\partial f(y)$.
	\end{itemize}
\end{lemma}

Before stating the K{\L} property, we first recall the definition of the following class of desingularizing functions.
\begin{definition}\label{def:desf}
	For $\zeta \in (0,+\infty]$, we define a class of desingularizing functions (denoted by $\varUpsilon_\zeta$) as the set of all continuous concave functions $\phi(\cdot) : [0,\zeta) \to \RR_+$ satisfying the following three properties:
	\begin{enumerate}
		\item [\rm (i)] $\phi(0) =0$;
		\item [\rm (ii)] $\phi(\cdot)$ is continuously differentiable on $(0,\zeta)$;
		\item [\rm (iii)] for all $s\in (0,\zeta)$, $\phi^\prime (s) >0$;
	\end{enumerate}
\end{definition}

For any subset $\mathbb{S}\subset \RR^n$ and any point $x\in\RR^n$, the distance from $x$ to $\mathbb{S}$, denoted by $\dist(x,\mathbb{S})$, is defined as 
$$\dist(x,\mathbb{S})=\inf\,\left\{ \| y-x\|\; |\; y\in\mathbb{S}\right\}.$$
With the above preparations, we now state the K{\L} property (e.g., see \cite[Definition 3.1]{ABRS10}).
\begin{definition}[K{\L} property, K{\L} exponent, and K{\L} function]\label{def:KL}
	Let $f(\cdot): \RR^n \to (-\infty,+\infty]$ be a proper and lower semicontinuous function.
	\begin{enumerate}
		\item [\rm (i)] We say that the function $f(\cdot)$ has the K{\L} property at 
		$$\bar{x} \in \dom{\partial f} := \left\{ x\in \RR^n\;|\; \partial f(x)  \ne \emptyset \right\}$$
		if there exist $\zeta \in (0,+\infty]$, a neighborhood $\mathcal{N}$ of $\bar{x}$, a continuous concave function $\phi(\cdot)\in \varUpsilon_\zeta$, and for all $x \in \mathcal{N} \cap \left\{y \in \RR^n\;|\; f(\bar{x}) < f(y) < f(\bar{x}) + \zeta \right\}$, the following K{\L} inequality holds, i.e.,
		\begin{equation}\label{eq:KL}
		\phi^\prime \left(f(x)-f(\bar{x})\right) \dist\left(0,\partial f(x)\right) \ge 1.
		\end{equation}
		\item [\rm (ii)] If $f(\cdot)$ satisfies the above K{\L} property at each point of $\dom{\partial f}$ and $\phi(\cdot)$ in \eqref{eq:KL} can be chosen as $\phi(s)=as^{1-\vartheta}$ for some $\vartheta\in[0,1)$ and $a>0$, then we say that $f(\cdot)$ satisfies the K{\L} property at $\bar x$ with exponent $\vartheta$. 
		\item[\rm (iii)] If $f(\cdot)$ satisfies the K{\L} property at all points in $\dom{\partial f}$, we say that $f(\cdot)$ is a K{\L} function. Moreover, if $f(\cdot)$ satisfies the K{\L} property with exponent $\vartheta\in[0,1)$ at all points in $\dom{\partial f}$, we say that $f(\cdot)$ is a K{\L} function with exponent $\vartheta$.
	\end{enumerate}
\end{definition}

To further analyze the convergence of the proposed algorithm to a critical point of \eqref{eq:problem}, we recall the uniformized K{\L} property introduced in \cite[Lemma 6]{BST14}, which is a more general K{\L} property.
\begin{lemma}[uniformized K{\L} property]\label{lem:UKL}
	Let $\Gamma$ be a compact set and let $f(\cdot): \RR^n \to (-\infty,+\infty]$ be a proper and lower semicontinuous function. Assume that $f(\cdot)$ is constant on $\Gamma$ and satisfies the K{\L} property at each point of $\Gamma$. Then, there exist $\varepsilon >0$, $\zeta>0$ and a continuous concave function $\phi\in \varUpsilon_\zeta$ such that for each $\widehat{x} \in \Gamma$ the following K{\L} inequality
	\begin{equation*}
	\phi^\prime\left(f(x)-f(\widehat{x}) \right) \dist \left(0,\partial f(x)\right) \ge 1,
	\end{equation*}
	holds for any
	$	x \in \left\{ y\in \RR^n\;|\; \dist (y,\Gamma) <\varepsilon \right\} \cap \left\{ y \in \RR^n\;|\; f(\widehat{x}) < f(y) < f(\widehat{x}) + \zeta\right\}.$
\end{lemma}

Interestingly, it is well documented in \cite{BDL07,BDLS07,BST14} that a broad class of real-world optimization problems satisfy the K{\L} property, since most of them have semi-algebraic objective functions or semi-algebraic constraint sets. Therefore, by the semi-algebraic family \cite[Example 2]{BST14}, we can see that all problems in Section \ref{Sec:Exp} satisfy the K{\L} property.

Below, we recall the famous descent lemma in optimization (e.g., see \cite[Lemma 5.7]{Beck17}).
\begin{lemma}[descent lemma]\label{lem:delem}
	Let $f(\cdot):\RR^n\to(-\infty,\infty]$ be an $L$-smooth function  ($L\geq 0$) 
	over a given convex set $\mathbb{S}$, i.e., $f(\cdot)$ is differentiable over $\mathbb{S}$ and satisfies
	$$\| \nabla f(x) -\nabla f(y)\| \leq L \|x-y\|, \quad \forall x,y \in \mathbb{S}.$$
	Then, for any $x,y\in\mathbb{S}$, we have 
	$$f(y)\leq f(x) + \langle \nabla f(x),y-x\rangle + \frac{L}{2}\|x-y\|^2.$$
\end{lemma}

\begin{definition}\label{def:conf}
	Let $f(\cdot):\RR^n \to [-\infty,+\infty]$ be an extended-real-valued function. The conjugate function of $f(\cdot)$ is defined by
	\begin{equation*}
	f^*(y) = \sup_{x\in \RR^n} \left\{ \langle y, x \rangle - f(x) \right\}.
	\end{equation*}
\end{definition}

It is well-know from \cite[Chapter 4]{Beck17} that if $f(\cdot)$ is proper, lower semicontinuous and convex, then $f^*(\cdot)$ is also proper, lower semicontinuous and convex. Furthermore, it is easy to see from Definition \ref{def:conf} that
\begin{equation}\label{eq:FY}
f(x)  + f^*(y) \ge \langle x,y \rangle,
\end{equation}
and in particular, the equality holds if and only if $y \in \partial f(x) $. Additionally, for any $x$ and $y$, one has $y \in \partial f(x) $ if and only if $x \in \partial f^*(y)$ (e.g., see \cite[Theorem 4.20]{Beck17}). In the literature, the inequality \eqref{eq:FY} is usually called Fenchel-Young inequality, which will be useful for our algorithmic design.

\begin{definition}[proximal mapping]\label{def:prox}
	Given a function $\theta(\cdot):\RR^n\to (-\infty,\infty]$ and a scalar $t>0$, the proximal mapping of $t\theta(\cdot)$ is the operator given by
	\begin{equation}\label{eq:prox}
	\prox_{t\theta}({\bm a})=\arg\min_{x\in\RR^n}\left\{t\theta(x)+\frac{1}{2}\|x-{\bm a}\|^2\right\} \quad \text{for any } {\bm a}\in\RR^n.
	\end{equation}
\end{definition}

When $\theta(\cdot)$ is a proper lower semicontinuous and convex function, the $\prox_{t\theta}(\cdot)$ is well-defined, i.e., the $\prox_{t\theta}({\bm a})$ is a singleton for any ${\bm a}\in\RR^n$.
In particular, when 
\begin{itemize}
	\item $\theta(x)=\mathcal{I}_{\mathbb{S}}(x)$, where $\mathcal{I}_{\mathbb{S}}(x)$ is the indicator function associated to the nonempty closed convex set $\mathbb{S}\subseteq \RR^n$, which is defined 
	$$\mathcal{I}_{\mathbb{S}}(x)=\left\{\begin{array}{ll} 0, & \text{ if } x\in\mathbb{S}, \\ \infty, & \text{ otherwise,}
	\end{array}\right.$$ then the proximal operator \eqref{eq:prox} reduces to the orthogonal projection onto $\mathbb{S}$, i.e.,
	$$\prox_{\mathcal{I}_{\mathbb{S}}}({\bm a})=\arg\min\left\{\|x-{\bm a}\|^2\;|\; x\in\mathbb{S}\right\}=\Proj_{\mathbb{S}}({\bm a}).$$
	\item $\theta(x)=\|x\|_1$, then the proximal operator \eqref{eq:prox} becomes the well-known shrinkage (or soft-thresholding) operator (see \cite{CW05}), i.e.,
	$$\prox_{t\|\cdot\|_1}({\bm a})=\shrink({\bm a},t)\equiv \text{\rm sign}({\bm a})\odot \max\{|{\bm a}|-t,0\},$$
	where `$\text{\rm sign}(\cdot)$' is the sign function, and `$\odot$' represents the component-wise product.
	\item $\theta(X)=\|X\|_*$ is the nuclear norm of matrix $X$, then the proximal operator \eqref{eq:prox} corresponds to the singular value thresholding operator (see \cite{CCS10}), i.e.,
	$$\prox_{t\|\cdot\|_*}(A)=\arg\min_{X}\left\{t\|X\|_*+\frac{1}{2}\|X-A\|^2_F\right\}=\SVT(A,t)\equiv U\shrink(\Sigma,t)V^\top,$$
	where $U\Sigma V^\top$ is the singular value decomposition of matrix $A$.
\end{itemize}

To end this section, we recall the definition of Bregman distance and its properties (see \cite{Bregman-1967} and also \cite[Chapter 9]{Beck17}), which are useful tools for our algorithmic design and convergence analysis.

\begin{definition}[Bregman distance]\label{def:Bregman}
	Let $\psi(\cdot): \RR^n \to (-\infty, + \infty]$ be a proper, closed and convex function that is differentiable over $\dom{\partial \psi}$. The Bregman distance associated with the kernel $\psi(\cdot)$ is the function $\B_\psi(\cdot,\cdot):\dom{\psi}\times \dom{\partial \psi}\to \RR$ given by
	\begin{equation*}
	\B_\psi(x,y) = \psi(x) - \psi(y) - \langle \nabla \psi(y), x-y\rangle,\quad \forall x \in \dom{\psi}, y \in \intset{\dom{\psi}}.
	\end{equation*}
\end{definition}

Generally speaking, the Bregman distance does not necessarily enjoy the symmetry and the triangle inequality property. However, such a measurement covers the standard Euclidean distance as its special case, i.e., $ \B_\psi(x,y)=\frac{1}{2}\|x-y\|^2$ if the kernel function is taken as $\psi(\cdot)=\frac{1}{2}\|\cdot\|^2$. More generally, letting $M$ be a positive definite matrix and $\psi(\cdot)=\frac{1}{2}\|\cdot\|_M^2$, we obtain $\B_\psi(x,y)=\frac{1}{2}\|x-y\|^2_M$, which is useful for solving quadratic programming. Besides, when taking the kernel function as $\psi(x)=\sum_{i=1}^n x_i \log x_i$, the Bregman distance is specified as 
$$\B_\psi(x,y)=\sum_{i=1}^n x_i \log \frac{x_i}{y_i} + y_i-x_i, \quad \forall x\in\RR^n_+,\;\forall y\in\RR^n_{++},$$
which plays an important role in solving optimization problems over unit simplex 
\begin{equation}\label{eq:simplex}
\Delta^n:=\left\{\;x\in\RR^n\;\Big{|}\; \sum_{i=1}^{n}x_i=1,\; x_i\geq 0,i=1,2,\ldots,n\;\right\},
\end{equation}
since such a Bregman function can simplify the proximal subproblems so that they enjoy closed-form solutions (see \cite{AT09} and our experiments in Section \ref{Sec:bid}). Below, we summarize some properties of Bregman distance (also see \cite{BBC03,Beck17}).

\begin{lemma}\label{lem:Bregman}
	Suppose that $\mathbb{S}\subseteq \RR^n$ is nonempty closed and convex, and the function $\psi(\cdot)$ is proper closed convex and differentiable over $\dom{\partial \psi}$. If $\mathbb{S}\subseteq \dom{\psi}$ and $\psi(\cdot)+\mathcal{I}_{\mathbb{S}}(\cdot)$ is $\varrho$-strongly convex ($\varrho>0$), then the Bregman distance $\B_\psi(\cdot,\cdot)$ associated with $\psi(\cdot)$ has the following properties:
	\begin{itemize}
		\item[\rm (i)] $\B_{\psi}(x,y)\geq \frac{\varrho}{2}\|x-y\|^2$ for all $x\in\mathbb{S}$ and $y\in\mathbb{S}\cap \dom{\partial \psi}$;
		\item[\rm (ii)] Let $x\in\mathbb{S}$ and $y\in\mathbb{S}\cap \dom{\partial \psi}$. Then $\B_{\psi}(x,y)\geq 0$, and in particular, the equality holds if and only if $x=y$;
		\item[\rm (iii)] If $\nabla \psi$ is Lipschitz continuous with modulus $L_\psi$, it holds that $\B_{\psi}(x,y)\leq \frac{L_\psi}{2}\|x-y\|^2$ for all $x\in\mathbb{S}$ and $y\in\mathbb{S}\cap \dom{\partial \psi}$.
	\end{itemize}
\end{lemma}


In the literature, the Bregman distance has been widely used to design efficient algorithms for optimization problems, we here refer the reader to \cite{BBC03,BR21,CZ92,Eck93,MOPS20,Teb18} and references therein for more properties and applications.

\section{The Unified Bregman Alternating Minimization Algorithm} \label{Sec:Alg}
In this section, we first introduce the unified Bregman alternating minimization algorithm for generalized DC programming \eqref{eq:problem}. Then, we will show that our algorithmic framework covers some state-of-the-art algorithms. 

As we know, those seminal DC algorithms for DC programming often require the second DC parts being differentiable. However, many real-world problems usually possess two nonsmooth DC components. In this paper, we follow the spirit of majorization minimization and then employ the Fenchel-Young inequality \eqref{eq:FY} to approximate $-g_1(x)$ and $-g_2(y)$ iteratively. More specifically, at the $k$-th iterate $(x^k,y^k)$, we always have 
$$g_1^*(\xi)-\langle x^k, \xi\rangle \geq -g_1(x^k)\quad \text{and}\quad g_2^*(\eta)-\langle y^k, \eta\rangle \geq -g_2(y^k).$$
Comparatively, the above approximations provide us a flexible way to exploit the favorable structure of the conjugate functions of $g_1(x)$ and $g_2(y)$ for some real-world applications. 

In addition, the coupling function $h(x,y)$ often makes the $x$- and $y$-subproblems difficult in the sense that both subproblems neither have closed-form solutions, even when the proximal operators of $f_1(x)$ and $f_2(y)$ can be expressed explicitly, nor can be accurately calculated via fast solvers. So, we employ the popular linearization strategy to approximate the underlying subproblems. However, we notice that the coupling function $h(x,y)$ is assumed to have two parts, i.e., $h(x,y)=h^+(x,y)-h^-(x,y)$. In particular, when $h(x,\cdot)$ (resp. $h(\cdot,y)$) is a DC function on $\RR^m$ (resp. $\RR^n$) for each $x\in\RR^n$ (resp. $y\in\RR^m$), we call $h(\cdot,\cdot)$ a generalized partial DC function (see \cite{PHLH22}). Here, we just assume that $h(x,y)$ consists of two parts, but do not strictly assume $h(x,y)$ being a generalized partial DC function. For this situation, directly linearizing $h(x,y)$ will ignore its favorable split nature, which encourages us to consider the linearization approximation to $h^-(x,y)$. Actually, we will show below that the split form of $h(x,y)$ is of benefit for designing customized variants of our algorithm. 

As shown in the literature, the uniqueness of optimal solutions to $x$- and $y$-subproblems is of importance for convergence analysis. Here we shall incorporate Bregman proximal regularization terms into the $x$- and $y$-subproblems, thereby ensuring that both subproblems are strongly convex as long as the proximal parameters can be chosen appropriately. Consequently, on the one hand, each subproblem has a unique solution. On the other hand, the Bregman proximal terms provide us a flexible framework to design customized algorithms, and meanwhile to understand some novel existing algorithms. Our unified algorithm is formally described in Algorithm \ref{alg:proposed}.

\begin{algorithm}
	\caption{Unified Bregman Alternating Minimization Algorithm for \eqref{eq:problem}.}\label{alg:proposed}
	\begin{algorithmic}
		\STATE{Take a starting point $(x^0,y^0) \in \RR^n \times \RR^m$, and strongly convex functions $\psi_k(\cdot)$ and $\varphi_k(\cdot)$.}
		\FOR{$k=0,1,2,\cdots$}
		\STATE{Compute $\xi^{k+1}$ and $\eta^{k+1}$:
			\begin{equation}\label{eq:subgradient}
			\xi^{k+1} \in \arg\min_{\xi \in \RR^n} \left\{g_1^*(\xi) - \langle x^k, \xi \rangle\right\} \quad \text{and}\quad \eta^{k+1}  \in \arg\min_{\eta \in \RR^m} \left\{ g_2^*(\eta) - \langle y^k,\eta \rangle\right\}.
			\end{equation}}
		\STATE{Update $x^{k+1}$ and $y^{k+1}$ in an alternating order as follows:
			\begin{align}
			x^{k+1}&= \arg\min_{x \in \RR^n} \left\{ f_1(x)+h^+(x,y^k)- \langle x-x^k,  \uu^k\rangle + \B_{\psi_k}(x,x^k)\right\}, \label{algeq:update_x}\\ 
			y^{k+1}& = \arg\min_{y \in \RR^m} \left\{f_2(y)+h^+(x^{k+1},y) -\langle y-y^k, \vv^k \rangle + \B_{\varphi_k}(y,y^k)\right\}, \label{algeq:update_y}
			\end{align}
			where $\uu^k=\xi^{k+1} + \nabla_x h^-(x^k,y^k)$ and $\vv^k=\eta^{k+1} + \nabla_y h^-(x^{k+1},y^k)$.
		}
		\ENDFOR
	\end{algorithmic}
\end{algorithm}

\begin{remark}\label{rem31}
Roughly speaking, both subproblems \eqref{algeq:update_x} and \eqref{algeq:update_y} dominate the main computational cost of Algorithm \ref{alg:proposed}. When taking some specific Bregman kernel functions such as $\psi_k(\cdot)=\varphi_k(\cdot)=\frac{1}{2}\|\cdot\|^2$, we can see that \eqref{algeq:update_x} and \eqref{algeq:update_y} amount to evaluating the following proximal operator:
\begin{equation*}
		\prox_{f+\hat{h}}({\bm a}^k)=\arg\min_{w}\left\{f(w)+\hat{h}(w)+\frac{1}{2}\|w-{\bm a}^k\|^2\right\},
\end{equation*}
where $f(\cdot)$ represents $f_1(\cdot)$ or $f_2(\cdot)$, $\hat{h}(\cdot)$ corresponds to $h^+(\cdot,y^k)$ or $h^+(x^{k+1},\cdot)$, and ${\bm a}^k$ is a point associated with $x^k$ or $y^k$. In this situation, the efficiency of Algorithm \ref{alg:proposed}, to a large extent, depends on the evaluation of the proximal operator $\prox_{f+\hat{h}}({\bm a}^k)$. Generally, the appearance of $\hat{h}(\cdot)$ coupling with a given $y^k$ or $x^{k+1}$ possibly leads to the proximal operator $\prox_{f+\hat{h}}({\bm a}^k)$ losing its closed-form solution. Therefore, how to maximally explore the explicit proximal operators is crucial for the application of Algorithm \ref{alg:proposed}. Indeed, we can easily handle the difficulty caused by $\hat{h}(\cdot)$ in many cases. Note that we do not impose strict limitations on $h^+(x,y)$ and $h^-(x,y)$. Therefore, we can freely rewrite the form of $h(x,y)$ according to actual needs such that $\hat{h}(\cdot)$ vanishes from $\prox_{f+\hat{h}}({\bm a}^k)$ for the purpose of exploiting the explicit form of $\prox_{f}({\bm a}^k)$ (e.g., see Remarks \ref{rem:sc4}-\ref{rem:sc3} and applications in Section \ref{Sec:Exp}). Of course, when the evaluation of proximal operator $\prox_{f+\hat{h}}({\bm a}^k)$ (or $\prox_{f}({\bm a}^k)$) is inevitable without a closed-form solution, we suggest employing some solvers equipped with inexact strategies (e.g., \cite{BTB23,Beck17}) to improve the performance of Algorithm \ref{alg:proposed}.
\end{remark}

Below, we present some remarks (i.e., Remarks \ref{rem:sc4}-\ref{rem:sc3}) to show that our Algorithm \ref{alg:proposed} can produce some existing novel iterative schemes, when reformulating some specific nonconvex optimization models as special cases of \eqref{eq:problem}. Here, we shall emphasize that one of the main goals of these remarks is to help us understand those state-of-the-art algorithms tailored for nonsmooth and nonconvex programming (e.g., \cite{ABRS10,BST14,NT19,PHLH22,SSC03,TFT22})  in a unified framework, so that those algorithms can be concatenated together via Algorithm \ref{alg:proposed}. Another main goal is to tell the reader how to design customized variants of Algorithm \ref{alg:proposed} for solving some real-world problems. So, we only show how to employ Algorithm \ref{alg:proposed} to derive those existing iterative schemes developed in the literature. It is noteworthy that our convergence results are not necessarily sufficient for the resulting algorithms, and the reader is referred to the original papers for their tightest convergence analysis or is requested to reanalyze their tighter convergence when the required conditions could be weakened.
 
\begin{remark}\label{rem:sc4}
	When the generalized DC programming \eqref{eq:problem} reduces to the standard DC programming of the form
	\begin{equation}\label{eq:sDC}
	\min_{x\in\RR^n} \left\{ \theta_1(x):=f_1(x)-g_1(x)\right\},
	\end{equation}
	it is clear that, by setting $\psi_k(x)=0$, an application of Algorithm \ref{alg:proposed} to \eqref{eq:sDC} yields the iterative scheme of the standard DC algorithm \cite{LTPD18,PDLT97}. When we take $\psi_k(x)=\frac{c_k}{2}\|x\|^2$, our Algorithm \ref{alg:proposed} for \eqref{eq:sDC} reads as
	\begin{equation*}
	x^{k+1}=\arg\min_{x\in\RR^n}\left\{f_1(x)-\langle x, \xi^{k+1}\rangle +\frac{c_k}{2}\|x-x^k\|^2\right\}\quad \text{ with }\quad \xi^{k+1}\in\partial g_1(x^k),
	\end{equation*}
	which immediately coincides with the iterative scheme of the method introduced in \cite{SSC03}. Moreover, when we further consider a more general case of \eqref{eq:sDC} as follows
	\begin{equation}\label{eq:gDC}
	\min_{x\in\RR^n} \left\{ \theta(x):=f_1(x)-g_1(x) + h_1(x)\right\},
	\end{equation}
	we can rewrite \eqref{eq:gDC} as 
	\begin{equation*}
	\min_{x\in\RR^n} \left\{ \Phi(x,y):=f_1(x)-g_1(x) - \left(- h_1(x) \right)\right\},
	\end{equation*}
	where the $h(x,y)$ in \eqref{eq:problem} can be specified as $h(x,y)=0- \left(- h_1(x) \right)$. In this situation, when assuming that $h_1(x)$ is differentiable, we have $\nabla_x h^-(x,y)=-\nabla h_1(x)$, and Algorithm \ref{alg:proposed} immediately reads as
\begin{equation*}
	x^{k+1}=\arg\min_{x\in\RR^n}\left\{f_1(x) + \langle x, \nabla h_1(x^k) -\xi^{k+1}\rangle +\B_{\psi_k}(x,x^k)\right\}\quad \text{ with }\quad \xi^{k+1}\in\partial g_1(x^k),
	\end{equation*}
	which corresponds to the iterative scheme of the method introduced in \cite{TFT22}. Particularly, we can also reformulate some constrained convex optimization problems as the form of \eqref{eq:gDC}, then the iterative schemes of some classical methods, such as proximal (projected) gradient method and mirror descent method (e.g., see \cite[Chapters 8-10]{Beck17}), can be derived under the framework of our Algorithm \ref{alg:proposed}.
\end{remark}

\begin{remark}\label{rem:sc1}
	When considering a simplified version of model \eqref{eq:problem} without the first two DC parts but with setting $h^-(x,y)=-f_1(x)-f_2(y)$ and assuming that $f_1(x)$ and $f_2(y)$ are continuously differentiable, model \eqref{eq:problem} is specified as
	\begin{equation*}
	\min_{x\in \RR^n, y\in \RR^m} \; \left\{\;\Phi(x,y) :=h^+(x,y)-\underbrace{\left( -f_1(x)-f_2(y)\right)}_{h^-(x,y)}\;\right\}.
	\end{equation*}
	Clearly, $\uu^k=\nabla_x h^-(x^k,y^k) = -\nabla f_1(x^k)$ and $\vv^k=\nabla_y h^-(x^{k+1},y^k)=-\nabla f_2(y^k)$. When the Bregman proximal terms $\B_{\psi_k}(x,x^k)$ and $\B_{\varphi_k}(y,y^k)$ are further taken as $\frac{1}{2\tau}\|x-x^k\|^2$ and $\frac{1}{2\sigma}\|y-y^k\|^2$ (i.e., by setting $\psi_k(x)=\frac{1}{2\tau}\|x\|^2$ and $\varphi_k(y)=\frac{1}{2\sigma}\|y\|^2$), respectively, our Algorithm \ref{alg:proposed} immediately is specified as
	\begin{equation*}
	\left\{\begin{aligned}
	x^{k+1}&= \arg\min_{x \in \RR^n} \left\{ h^+(x,y^k)+ \langle x,  \nabla f_1(x^k) \rangle + \frac{1}{2\tau}\|x-x^k\|^2 \right\}, \\ 
	y^{k+1}& = \arg\min_{y \in \RR^m} \left\{h^+(x^{k+1},y) +\langle y, \nabla f_2(y^k) \rangle + \frac{1}{2\sigma}\|y-y^k\|^2\right\},  
	\end{aligned}\right.
	\end{equation*}
	which is precisely the same as the iterative scheme of the so-named Alternating Structured-Adapted Proximal Gradient Descent (ASAPGD) algorithm proposed in \cite{NT19}. Actually, we can derive some other variants of the ASAPGD algorithm by taking different Bregman proximal terms so that the potentially favorable structure of $h^+(x,y)$ can be efficiently exploited.
\end{remark}

\begin{remark}\label{rem:sc2}
	As discussed in Remark \ref{rem:sc1}, when we consider the case of \eqref{eq:problem} with the form
	\begin{equation*}\label{eq:pv2}
	\min_{x\in \RR^n, y\in \RR^m} \; \left\{\;h(x,y) :=h^+(x,y)-h^-(x,y)\;\right\},
	\end{equation*}
	and assume $h(x,y)$ being a generalized partial DC function, our Algorithm \ref{alg:proposed} immediately reduces to 
	\begin{equation*}
	\left\{\begin{aligned}
	x^{k+1}&= \arg\min_{x \in \RR^n} \left\{ h^+(x,y^k)- \langle x,   \nabla_x h^-(x^k,y^k)\rangle + \B_{\psi_k}(x,x^k)\right\}, \\ 
	y^{k+1}& = \arg\min_{y \in \RR^m} \left\{h^+(x^{k+1},y) -\langle y, \nabla_y h^-(x^{k+1},y^k) \rangle + \B_{\varphi_k}(y,y^k)\right\}.
	\end{aligned}\right.
	\end{equation*}
	Apparently, when setting $\psi_k(x)=0$ and $\varphi_k(y)=0$, the resulting iterative scheme is essentially the same as the one proposed in the most recent work \cite{PHLH22} when dealing with differentiable cases. More generally, reformulating our model \eqref{eq:problem} as
	\begin{equation}\label{eq:gdc}
	\min_{x\in \RR^n, y\in \RR^m} \; \left\{\;\Phi(x,y) :=\underbrace{\left(f_1(x)+f_2(y)+h^+(x,y)\right)}_{h_1(x,y)}-\underbrace{\left(g_1(x)+g_2(y)+h^-(x,y)\right)}_{h_2(x,y)}\;\right\},
	\end{equation}
	then applying the alternating DC algorithm \cite{PHLH22} to \eqref{eq:gdc} also falls into the special case of Algorithm \ref{alg:proposed} without the Bregman proximal terms. Comparatively, our Algorithm \ref{alg:proposed} possesses two universal proximal regularization terms so that the underlying subproblems usually have a unique solution.
\end{remark}

\begin{remark}\label{rem:sc3}
	Interestingly, when considering the special case of \eqref{eq:problem} where $g_1(x)$ and $g_2(y)$ are zeros, we can easily show that our Algorithm \ref{alg:proposed} covers the iterative schemes introduced in \cite{ABRS10} (i.e., \eqref{eq:pAM}) and \cite{BST14} (i.e., \eqref{eq:PALM}) via wisely setting $h(x,y)$ (Here we should mention again that we do not strictly require $f_1(\cdot)$ and $f_2(\cdot)$ being convex, which is a preparation to establish the connection of our algorithm to others). Specifically,
	\begin{itemize}
		\item when $h(x,y)$ is simplified as $h(x,y)=h^+(x,y)$, model \eqref{eq:problem} reduces to 
		\begin{equation*}
		\min_{x\in \RR^n, y\in \RR^m} \; \left\{\;\Phi(x,y) :=f_1(x)+f_2(y) +h(x,y) \;\right\}.
		\end{equation*}
		By taking $\psi_k(x) = \frac{c_k}{2}\|x\|^2$ and $\varphi_k(y) = \frac{d_k}{2}\|y\|^2$, Algorithm \ref{alg:proposed} becomes 
		\begin{equation*}
		\left\{\begin{aligned}
		x^{k+1}&= \arg\min_{x \in \RR^n} \left\{ f_1(x) +h(x,y^k) + \frac{c_k}{2}\|x-x^k\|^2\right\}, \\ 
		y^{k+1}& = \arg\min_{y \in \RR^m} \left\{f_2(y) + h(x^{k+1},y) +  \frac{d_k}{2}\|y-y^k\|^2\right\},
		\end{aligned}\right.
		\end{equation*}
		which is the algorithm introduced in \cite{ABRS10}.
		\item when $h(x,y)$ is simplified as $h(x,y)=-h^-(x,y)$, model \eqref{eq:problem} reduces to 
		\begin{equation*}
		\min_{x\in \RR^n, y\in \RR^m} \; \left\{\;\Phi(x,y) :=f_1(x)+f_2(y) - \left(-h(x,y)\right)\;\right\}.
		\end{equation*}
		Similarly, by setting $\psi_k(x) = \frac{c_k}{2}\|x\|^2$ and $\varphi_k(y) = \frac{d_k}{2}\|y\|^2$, Algorithm \ref{alg:proposed} reduces to
		\begin{equation*}
		\left\{\begin{aligned}
		x^{k+1}&= \arg\min_{x \in \RR^n} \left\{ f_1(x) +\langle x,   \nabla_x h(x^k,y^k)\rangle + \frac{c_k}{2}\|x-x^k\|^2\right\}, \\ 
		y^{k+1}& = \arg\min_{y \in \RR^m} \left\{f_2(y) + \langle y, \nabla_y h(x^{k+1},y^k) \rangle +  \frac{d_k}{2}\|y-y^k\|^2\right\},
		\end{aligned}\right.
		\end{equation*}
		which is precisely the same as the PALM algorithm proposed in \cite{BST14}.
	\end{itemize}
	Actually, as discussed above, the flexibility of these Bregman proximal terms allows us to derive more variants under our unified framework.
\end{remark}

\section{Convergence analysis}\label{Sec:Converg}
In this section, we are concerned with the convergence properties of the proposed Algorithm \ref{alg:proposed}. We begin this section with stating some standard assumptions on model \eqref{eq:problem} as assumed in \cite{PS16}, which will be used in convergence analysis.

\begin{assumption} \label{ass:1}
	The components $f_i $ $(i=1,2) $ are proper lower semicontinuous functions, and $g_i $ $(i=1,2)$ are continuous and convex functions.
\end{assumption}

\begin{assumption} \label{ass:2}
	$\inf_{x,y} \; \Phi(x,y) = \inf_{x,y} \; \left\{\theta_1(x)+  \theta_2(y) + h(x,y) \right\}  > -\infty.$
\end{assumption}

\begin{assumption}\label{ass:3}
	The coupling function $h(x,y)= h^+(x,y) - h^-(x,y)$ satisfies the following conditions.
	\begin{itemize}
		\item[\rm (i)] For any fixed $y$, the function $h^-(\cdot,y)$ is $C_{L_1^-(y)}^{1,1}$, namely the partial gradient $\nabla_x h^-(\cdot,y)$ is globally Lipschitz with moduli $L_1^-(y)$, that is
		\begin{equation*}
		\lno \nabla_x h^-(x_1,y) - \nabla_x h^-(x_2,y) \rno \le L_1^-(y)\|x_1-x_2\|, \; \forall x_1,x_2\in\RR^n.
		\end{equation*}
		Likewise, for any fixed $x$, the function $h^-(x,\cdot)$ is assumed to be $C_{L_2^-(x)}^{1,1}$.
		\item[\rm (ii)] For $i=1,2$, there exist $\lambda_i^-$ and $\lambda_i^+>0$ such that
		\begin{equation} \label{eq:Lip-lb}
		\inf\{L_1^-(y^k): k\in \NN \} \ge \lambda_1^- \;\;\text{ and }\;\; \inf\{L_2^-(x^k): k\in \NN \} \ge \lambda_2^-,
		\end{equation}
		and
		\begin{equation} \label{eq:Lip-ub}
		\sup\{L^-_1(y^k): k\in \NN \} \le \lambda_1^+\;\; \text{ and }\;\; \sup\{L_2^-(x^k): k\in \NN \} \le \lambda_2^+.
		\end{equation}
		\item[\rm (iii)] $h(x,y)$ is a continuously differentiable function and its gradient $\nabla h(x,y) = \nabla h^+(x,y)  - \nabla h^-(x,y) $ is Lipschitz continuous on bounded subsets of $\RR^n \times \RR^m$. That is, for each bounded subsets $\mathbb{B}_1 \times \mathbb{B}_2 $ of $\RR^n \times \RR^m$, there exists $L>0$ (or $L^+$ and $L^-$ satisfying $L=L^+ + L^-$) such that
		\begin{equation*}
		\lno \begin{pmatrix} \nabla_x h (x_1,y_1) \\  \nabla_y h (x_1,y_1) \end{pmatrix} - \begin{pmatrix} \nabla_x h(x_2,y_2) \\ \nabla_y h (x_2,y_2) \end{pmatrix}\rno 
		\le L \lno \begin{pmatrix} x_1\\ y_1\end{pmatrix}-\begin{pmatrix} x_2\\ y_2\end{pmatrix} \rno
		\end{equation*}
		for all $(x_i,y_i) \in \mathbb{B}_1 \times \mathbb{B}_2$ $(i =1,2)$.
	\end{itemize}
\end{assumption}

Before the convergence analysis, for notational simplicity, we let
\begin{equation} \label{eq:auxfun}
\Psi(x,\xi,y,\eta) = f_1(x) + g_1^*(\xi) - \langle x,\xi \rangle + f_2(y) + g_2^*(\eta) - \langle y,\eta \rangle + h(x,y),
\end{equation}
which can be viewed as a surrogate objective function of \eqref{eq:problem} to approximate $\Phi(x,y)$. It then follows from the Fenchel-Young inequality \eqref{eq:FY} that
\begin{equation}\label{eq:PsiPhi}
\Psi(x,\xi,y,\eta) \ge \Phi(x,y).
\end{equation}
In what follows, we denote $\ww:=(x,\xi,y,\eta)$ and $\zz:=(x,y)$ for notational simplicity. Besides, we use $\B_{\psi_k}(x,x^k)=c_k\B_{\psi}(x,x^k)$ and $\B_{\varphi_k}(y,y^k)=d_k\B_{\varphi}(y,y^k)$ to simplify notations in convergence analysis, where $c_k$ and $d_k$ are iteration-varying parameters, both $\psi(\cdot)$ and $\varphi(\cdot)$ are parameter-free Bregman kernel functions. Therefore, we use $\psi_k(\cdot)=c_k\psi(\cdot)$ and $\varphi_k(\cdot)=d_k\varphi(\cdot)$ throughout this paper, and each iteration-varying parameter and the strongly convex constant of the corresponding Bregman kernel function will be merged into one for simplicity. Now, we first show that $\{\Psi(x^k,\xi^k,y^k,\eta^k)\}\equiv \{\Psi(\ww^k)\} $ and $\{\Phi(x^k,y^k)\}\equiv \{\Phi(\zz^k)\}$ are decreasing sequences as $k\to \infty$.

\begin{lemma}\label{lem:descent}
	Suppose Assumptions \ref{ass:1} and \ref{ass:3} hold. Let $\left\{\zz^k := (x^k,y^k) \right\}$ be the sequence generated by Algorithm~\ref{alg:proposed}. Then, both sequences $\left\{\Psi(\ww^k)\right\}$ and $\left\{\Phi(\zz^k)\right\}$ are decreasing, and there exists a constant $\gamma_k>0$ such that
	\begin{equation} \label{eq:dec_cond}
	\Psi(\ww^{k+1}) \leq \Psi(\ww^k) - \frac{\gamma_k}{2}\lno \zz^{k+1}-\zz^k\rno^2.
	\end{equation}
\end{lemma}

\begin{proof}
	Invoking the first-order optimality condition of \eqref{eq:subgradient} yields
	\begin{equation} \label{eq:optcond_xi}
	0 \in \partial g_1^*(\xi^{k+1}) - x^k \quad \text{and}\quad 
	0 \in \partial g_2^*(\eta^{k+1}) - y^k.
	\end{equation}
	It then follows from the Fenchel-Young inequality \eqref{eq:FY} that 
	\begin{equation}\label{eq:optcond_xi_FY}
	\xi^{k+1} \in \partial g_1(x^k) \quad \text{and}\quad 
	\eta^{k+1} \in \partial g_2(y^k),
	\end{equation}
	which further implies that
	\begin{equation}\label{eq:conx}
	g_1^*(\xi^{k+1}) = \langle x^k, \xi^{k+1} \rangle - g_1(x^k) \quad \text{and} \quad
	g_2^*(\eta^{k+1}) = \langle y^k, \eta^{k+1} \rangle - g_2(y^k).
	\end{equation}
	Consequently, it follows from \eqref{eq:conx} and the notion of $\Psi(\ww)$ given by \eqref{eq:auxfun} that
	\begin{align}
	\Psi(\ww^{k+1}) & = f_1(x^{k+1}) + g_1^*(\xi^{k+1}) - \langle x^{k+1},\xi^{k+1} \rangle + f_2(y^{k+1}) + g_2^*(\eta^{k+1}) - \langle y^{k+1},\eta^{k+1} \rangle \nonumber \\
	&\qquad + h^+(x^{k+1},y^{k+1}) -h^-(x^{k+1},y^{k+1}) \nonumber \\
	&=  f_1(x^{k+1}) - \langle x^{k+1}-x^k, \xi^{k+1} \rangle -g_1(x^k) + f_2(y^{k+1}) - \langle y^{k+1}-y^k,\eta^{k+1} \rangle - g_2(y^k) \label{eq:lem1-1}\\
	&\qquad   + h^+(x^{k+1},y^{k+1}) -h^-(x^{k+1},y^{k+1}). \nonumber
	\end{align}
	According to Assumption~\ref{ass:3} (i), it follows from Lemma \ref{lem:delem} that
	\begin{equation*}
	- h^-(x^{k+1},y^k) \le - h^-(x^k,y^k)  - \left\langle x^{k+1}-x^k , \nabla_x h^-(x^k,y^k) \right\rangle + \frac{L_1^-(y^k)}{2} \lno x^{k+1} - x^k \rno^2
	\end{equation*}
	where $L_1^-(y^k)$ is the Lipschitz constant of $\nabla_x h^-(x,y^k)$ at $x^k$. Similarly, we have
	\begin{equation*}
	- h^-(x^{k+1},y^{k+1}) \le - h^-(x^{k+1},y^k)  - \left\langle y^{k+1}-y^k , \nabla_y h^-(x^{k+1},y^k) \right\rangle + \frac{L_2^-(x^{k+1})}{2} \lno y^{k+1} - y^k \rno^2,
	\end{equation*}
	where $L_2^-(x^{k+1})$ is the Lipschitz constant of $\nabla_y h^-(x^{k+1},y)$ at $y^k$. Consequently, substituting the above two inequalities into \eqref{eq:lem1-1} leads to
	\begin{align}\label{proofeq:linearized}
	\Psi(\ww^{k+1}) &\leq  f_1(x^{k+1})  -\left\langle x^{k+1}-x^k, \xi^{k+1} + \nabla_x h^-(x^k,y^k) \right\rangle - g_1(x^k)   \\
	& \quad+  f_2(y^{k+1})  - \left\langle y^{k+1}-y^k,\eta^{k+1} + \nabla_y h^-(x^{k+1},y^k)\right\rangle  - g_2(y^k)\nonumber \\
	& \quad  +\frac{L_1^-(y^k)}{2} \lno x^{k+1} - x^k \rno^2 + \frac{L_2^-(x^{k+1})}{2}\lno y^{k+1} - y^k \rno^2 \nonumber \\
	&\quad + h^+(x^{k+1},y^{k+1})- h^-(x^k,y^k).    \nonumber
	\end{align}
	Using the updating schemes of $x^{k+1}$ and $y^{k+1}$ in~\eqref{algeq:update_x} and~\eqref{algeq:update_y}, respectively, we have
	\begin{align}\label{proofeq:update_x}
	f_1(x^{k+1}) + h^+(x^{k+1},y^k) - &\left\langle x^{k+1}-x^k, \xi^{k+1} + \nabla_x h^-(x^k,y^k) \right\rangle \\
	&\hskip 1cm + \B_{\psi_k}(x^{k+1},x^k)\le f_1(x^k) + h^+(x^k,y^k) \nonumber
	\end{align}
	and 
	\begin{align}\label{proofeq:update_y}
	f_2(y^{k+1}) + h^+(x^{k+1},y^{k+1}) - &\left\langle y^{k+1}-y^k,\eta^{k+1} + \nabla_y h^-(x^{k+1},y^k)\right\rangle \\
	&\hskip 1cm	+ \B_{\varphi_k}(y^{k+1},y^k)\le f_2(y^k) + h^+(x^{k+1},y^k). \nonumber
	\end{align}
	Plugging the two inequalities~\eqref{proofeq:update_x} and~\eqref{proofeq:update_y} into~\eqref{proofeq:linearized} yields
	\begin{align}\label{eq:lem1-2}
	\Psi(\ww^{k+1})& \leq f_1(x^k) - g_1(x^k) +  f_2(y^k)  - g_2(y^k) + h^+(x^k,y^k) - h^-(x^k,y^k) - \B_{\psi_k}(x^{k+1},x^k)\nonumber \\
	&\qquad  + \frac{L_1^-(y^k)}{2} \lno x^{k+1} - x^k \rno^2   + \frac{L_2^-(x^{k+1})}{2}\lno y^{k+1} - y^k \rno^2  - \B_{\varphi_k}(y^{k+1},y^k)  \nonumber\\
	&= \Phi(\zz^k) - \B_{\psi_k}(x^{k+1},x^k) + \frac{L_1^-(y^k)}{2}\lno x^{k+1} - x^k\rno^2  \\
	&\hskip1.5cm - \B_{\varphi_k}(y^{k+1},y^k) + \frac{L_2^-(x^{k+1})}{2}\lno y^{k+1} - y^k\rno^2. \nonumber
	\end{align}
	Recalling the strong convexity of $\psi_k(\cdot)$ and $\varphi_k(\cdot)$ and Lemma \ref{lem:strong}, we accordingly chose $\psi_k(\cdot)$ and $\varphi_k(\cdot)$ with strongly convex modulus $\rho_1$ and $\rho_2$, respectively, so that $\widetilde{\psi}_k(\cdot)= \psi_k(\cdot) - \frac{L_1^-(y^k)}{2} \|\cdot\|^2$ and $\widetilde{\varphi}_k(\cdot) = \varphi_k(\cdot) - \frac{L_2^-(x^{k+1})}{2}\|\cdot\|^2$ are still strongly convex. Let
	\begin{numcases}{}
	\B_{\widetilde{\psi}_k}(x^{k+1},x^k) :=\B_{\psi_k}(x^{k+1},x^k) - \frac{L_1^-(y^k)}{2}\lno x^{k+1} - x^k\rno^2, \nonumber \\ \B_{\widetilde{\varphi}_k}(y^{k+1},y^k) :=\B_{\varphi_k}(y^{k+1},y^k) - \frac{L_2^-(x^{k+1})}{2}\lno y^{k+1} - y^k\rno^2,  \nonumber
	\end{numcases}
	and $\widehat{\B_k}(\zz^{k+1},\zz^k) := \B_{\widetilde{\psi}_k}(x^{k+1},x^k) + \B_{\widetilde{\varphi}_k}(y^{k+1},y^k)$.
	Therefore, it follows from the strong convexity of $\widetilde{\psi}_k(\cdot)$ and $\widetilde{\varphi}_k(\cdot)$ that there exists $\gamma_k:=\min\{\rho_1-L_1^-(y^k), \rho_2-L_2^-(x^{k+1})\}>0$ such that
	\begin{equation}\label{eq:Bzk}
	\widehat{\B_k}(\zz^{k+1},\zz^k) \ge \frac{\gamma_k}{2} \left\|\zz^{k+1} - \zz^k \right\|^2.
	\end{equation}
	Therefore, combining \eqref{eq:lem1-2} and \eqref{eq:Bzk} arrives at
	\begin{align}\label{ineq:lem1}
	\Psi(\ww^{k+1}) 
	&\leq  \Phi(\zz^k) - \left( \B_{\widetilde{\psi}_k}(x^{k+1},x^k) +  \B_{\widetilde{\varphi}_k}(y^{k+1},y^k)\right) \\
	& =  \Phi(\zz^k) - \widehat{\B_k}(\zz^{k+1},\zz^k)  \nonumber \\
	&	\le  \Phi(\zz^k) - \frac{\gamma_k}{2}\lno \zz^{k+1}-\zz^k\rno^2 \nonumber\\
	&	\le  \Psi(\ww^k) - \frac{\gamma_k}{2}\lno \zz^{k+1}-\zz^k\rno^2. \nonumber
	\end{align}
	where the last inequality follows from \eqref{eq:PsiPhi}. Such an inequality immediately means that the sequence $\{\Psi(\ww^k)\}$ is decreasing. Moreover, using \eqref{eq:PsiPhi} again to the left hand side of \eqref{ineq:lem1} yields
	\begin{equation}\label{eq:nondphi}
	\Phi(\zz^{k+1})\leq \Psi(\ww^{k+1}) \leq \Phi(\zz^k) - \frac{\gamma_k}{2}\lno \zz^{k+1}-\zz^k\rno^2,
	\end{equation}
	which shows that $\{\Phi(\zz^k)\}$ is also decreasing. The proof is completed.
\end{proof}

Hereafter, we will show that two adjacent points of the sequence $\left\{ \zz^k=(x^k,y^k) \right\}$ generated by Algorithm~\ref{alg:proposed} will get infinitely close as $k\to\infty$, which is a key property for the convergence.
 
\begin{lemma}\label{lem:bound}
	Suppose that Assumptions \ref{ass:1}-\ref{ass:3} hold. Let $\left\{ \zz^k=(x^k,y^k) \right\}$ be the sequence generated by Algorithm~\ref{alg:proposed}. Then, we have 
	\begin{equation*}
	\sum_{k=0}^\infty \lno \zz^{k+1} - \zz^k \rno^2 < \infty,
	\end{equation*}
	and hence $\lim_{k\to \infty} \|\zz^{k+1}-\zz^k\|=0$.
\end{lemma}

\begin{proof}
	According to \eqref{eq:Lip-lb}, we can expect $\mu= \inf_k \{\gamma_k\} > 0$. Then, it follows from \eqref{eq:nondphi} that
	\begin{equation} \label{eq:proof_dec_cond}
	\Phi(\zz^{k+1}) \le \Phi(\zz^k) - \frac{\gamma_k}{2}\lno \zz^{k+1}-\zz^k\rno^2  \le \Phi(\zz^k) - \frac{\mu}{2}\lno \zz^{k+1}-\zz^k\rno^2
	\end{equation}
	Let $N$ be a positive integer. Summing \eqref{eq:proof_dec_cond} from $k=0$ to $N-1$ leads to
	\begin{equation*}
	\sum_{k=0}^{N-1} \lno \zz^{k+1} - \zz^k \rno^2 \le \frac{2}{\mu}\left(\Phi(\zz^0) - \Phi(\zz^N)\right).
	\end{equation*}
	By Assumption \ref{ass:2}, i.e., $\inf_{\zz} \{\Phi(\zz):=\Phi(x,y)\} > -\infty$, taking the limit as $N \to \infty$ immediately arrives at
	\begin{equation*}
	\sum_{k=0}^\infty \lno \zz^{k+1} - \zz^k \rno^2 < \infty.
	\end{equation*}
	Hence, we conclude that $\lim_{k\to \infty} \|\zz^{k+1}-\zz^k\|=0$.
\end{proof}

Below, we present an inequality characterizing the relationship between the sequences $\{\ww^k\}$ and $\{\zz^k\}$, which is the pivot inequality to prove the sequence approaching to a critical point of \eqref{eq:problem}.
\begin{lemma}\label{lem:disbound}
	Suppose Assumptions \ref{ass:1} and \ref{ass:3} hold. Let $\left\{\zz^k := (x^k,y^k) \right\}$ be the sequence generated by Algorithm~\ref{alg:proposed} which is assumed to be bounded. Then, there exists $\tau>0$ such that for any $k \in \NN$, we have
	\begin{equation*}
	\dist(0,\partial \Psi(\ww^{k+1})) \le \tau \lno \zz^{k+1} - \zz^k\rno.
	\end{equation*}
\end{lemma}

\begin{proof}
	Writing the first-order optimality conditions of~\eqref{algeq:update_x} and~\eqref{algeq:update_y} arrives at
	\begin{equation}\label{eq:optcond_x}
	0 \in \partial f_1(x^{k+1}) - \xi^{k+1}+ \nabla_x h^+(x^{k+1},y^k)  - \nabla_x h^-(x^k,y^k) + \nabla \psi(x^{k+1}) - \nabla \psi(x^k)
	\end{equation}
	and
	\begin{equation}\label{eq:optcond_y}
	0 \in \partial f_2(y^{k+1})  -\eta^{k+1}+ \nabla_y h^+(x^{k+1},y^{k+1}) - \nabla_y h^-(x^{k+1},y^k) + \nabla \varphi(y^{k+1}) - \nabla \varphi(y^k) .
	\end{equation}
	On the other hand, calculating $\partial \Psi(\ww)$ at $\ww^{k+1}:= (x^{k+1},\xi^{k+1},y^{k+1},\eta^{k+1})$ leads to
	\begin{equation}\label{eqs:subgrads_E}
	\left\{\begin{aligned}
	\partial f_1(x^{k+1}) - \xi^{k+1} + \nabla_xh(x^{k+1},y^{k+1}) &= \partial_x \Psi(\ww^{k+1}), \\
	-x^{k+1} + \partial g_1^*(\xi^{k+1}) &= \partial_{\xi} \Psi(\ww^{k+1}) ,\\
	\partial f_2(y^{k+1}) - \eta^{k+1} + \nabla_yh(x^{k+1},y^{k+1}) &= \partial_y \Psi(\ww^{k+1}) ,\\
	-y^{k+1} + \partial g_2^*(\eta^{k+1}) &= \partial_{\eta} \Psi(\ww^{k+1}).
	\end{aligned}\right.
	\end{equation}
	Consequently, substituting~\eqref{eq:optcond_x} and \eqref{eq:optcond_y} into~\eqref{eqs:subgrads_E} and using the fact~\eqref{eq:optcond_xi} yields
	\begin{align*}
	& \dist(0,\partial \Psi(\ww^{k+1})) \\
	& \le \lno \nabla_x h^+(x^{k+1},y^{k+1}) - \nabla_x h^+(x^{k+1},y^k) \rno + \lno \nabla_x h^-(x^{k+1},y^{k+1}) - \nabla_x h^-(x^k,y^k) \rno \\
	& \quad + \lno \nabla \psi(x^{k+1}) - \nabla \psi(x^k) \rno + \lno x^{k+1}-x^k \rno  + \lno \nabla_y h^-(x^{k+1},y^{k+1}) - \nabla_y h^-(x^{k+1},y^k) \rno \\
	& \quad + \lno \nabla \varphi(y^{k+1}) - \nabla \varphi(y^k) \rno + \lno y^{k+1}-y^k \rno \\
	& \le (L_+ + 2L_- + 2 + L_{\psi_k}+L_{\varphi_k}) \lno \zz^{k+1}-\zz^k\rno,
	\end{align*}
	where the first inequality follows from the definition of distance function and the triangle inequality (i.e., $\|a+b\|\leq \|a\| + \|b\|$ for all $a,b\in\RR^n$), and the second inequality is due to the Lipschitz continuity of $\nabla \psi_k$ and $\nabla \varphi_k$ with Lipschitz constants $L_{\psi_k}$ and $L_{\varphi_k}$, respectively. We can take $\tau = \sup\{ L_+ + 2L_- +1 + L_{\psi_k}+L_{\varphi_k}:k\in \NN\} < \infty$ from \eqref{eq:Lip-ub} and immediately conclude that
	\begin{equation}\label{eq:err_cond}
	\dist(0,\partial \Psi(\ww^{k+1})) \le \tau \lno \zz^{k+1}-\zz^k\rno.
	\end{equation}
	The proof is complete.
\end{proof}

Note that the boundedness of the sequence is a standard assumption for many nonconvex optimization algorithms. It is documented in \cite[Remark 3.3]{ABRS10} that the boundedness assumption on the sequence $\left\{\zz^k \right\}$ automatically holds when the corresponding lower level set $\{\zz\;|\; \Phi(\zz) \le \alpha_0\}$ is compact for some $\alpha_0\in \RR$ and $h(x,y)=\frac{1}{2}\|x-y\|^2$. Although checking the aforementioned boundedness is not an easy task in general, a large number of nonconvex optimization algorithms have received great successes in the fields of data sciences, machine learning and image processing. To a certain extent, those successful applications imply that the required conditions such as the boundedness of the sequences generated by nonconvex optimization algorithms could be reached with a high probability, or some theoretical assumptions could be further weakened.

\begin{lemma}\label{prop:2}
	Suppose that Assumptions \ref{ass:1}--\ref{ass:3} hold. Let $\left\{\zz^k:=(x^k,y^k) \right\}$ be a sequence generated by Algorithm \ref{alg:proposed} which is assumed to be bounded. The following assertions hold.
	\begin{enumerate}
		\item [\rm (i)] The set of accumulation points of the sequence $\left\{\ww^k:=(x^k,\xi^k,y^k,\eta^k) \right\}$, denoted by $\mathbb{W}^*$, is a nonempty compact set and contained in $\crit{\Psi}$, i.e., $\mathbb{W}^* \subset \crit{\Psi}$.
		\item [\rm (ii)] The limit $\Psi^{\infty} := \displaystyle \lim_{k\to \infty} \Psi(\ww^k)$ exists.
		\item [\rm (iii)] $\Psi(\ww) \equiv \Psi^{\infty}$ for all $\ww\in\mathbb{W}^*$.
	\end{enumerate}
\end{lemma}

\begin{proof}
	We first prove Item (i). The boundedness assumption on $\{\zz^k=(x^k,y^k)\}$, which, together with the continuity and convexity of $g_i(\cdot)$ ($i=1,2$), implies that the sequence $\{(\xi^k,\eta^k)\}$ is also bounded. As a consequence, the sequence $\left\{\ww^k=(x^k,\xi^k,y^k,\eta^k) \right\}$ is also bounded, which, together with the closedness of the set $\mathbb{W}^*$, immediately shows that $\mathbb{W}^*$ is a nonempty compact set. Given any $\ww^*\in \mathbb{W}^*$ being a limit point of some subsequence, denoted by $\left\{\ww^{k_j}=(x^{k_j},\xi^{k_j},y^{k_j},\eta^{k_j})\right\}$, we will prove that $\ww^*=(x^*,\xi^*,y^*,\eta^*)$ is a critical point of $\Psi(\cdot)$.
	Letting $k=k_j-1$, it follows from the optimal condition of \eqref{algeq:update_x} that
	\begin{align*}
	f_1(x^{k_j}) \;+\;& h^+(x^{k_j},y^{k_j-1}) - \langle x^{k_j}-x^{k_j-1}, \uu^{k_j-1}\rangle + \B_{\psi_k}(x^{k_j}, x^{k_j-1}) \\
	&\le f_1(x^*) + h^+(x^*,y^{k_j-1}) - \langle x^*-x^{k_j-1}, \uu^{k_j-1}\rangle+ \B_{\psi_k}(x^*, x^{k_j-1}),
	\end{align*}
	which is equivalent to 
	\begin{align}\label{eqs:optcond.x}
	f_1(x^{k_j}) \;+\;& h^+(x^{k_j},y^{k_j-1}) + \B_{\psi_k}(x^{k_j}, x^{k_j-1})  \\ 
	&\le  f_1(x^*) + h^+(x^*,y^{k_j-1}) - \langle x^*-x^{k_j}, \uu^{k_j-1}\rangle + \B_{\psi_k}(x^*, x^{k_j-1}). \nonumber 
	\end{align}	
	Similarly, it follows from the optimal condition of \eqref{algeq:update_y} that
	\begin{align} \label{eqs:optcond.xy}
	f_2(y^{k_j}) \;+\;&  h^+(x^{k_j},y^{k_j}) + \B_{\varphi_k}(y^{k_j}, y^{k_j-1}) \\
	&\leq f_2(y^*) + h^+(x^{k_j},y^*) - \langle y^*-y^{k_j}, \vv^{k_j-1}\rangle + \B_{\varphi_k}(y^*, y^{k_j-1}). \nonumber
	\end{align}
	Hence, it follows from the nonnegativity of Bregman distance and the two inequalities \eqref{eqs:optcond.x} and \eqref{eqs:optcond.xy} that
	\begin{align*}
	\lim_{j\to \infty} \Psi(\ww^{k_j}) 
	&=  \lim_{j\to \infty} f_1(x^{k_j}) - \langle x^{k_j}, \xi^{k_j} \rangle + g_1^*(\xi^{k_j}) + f_2(y^{k_j}) - \langle y^{k_j}, \eta^{k_j} \rangle + g_2^*(\eta^{k_j}) + h(x^{k_j},y^{k_j}) \\
	&\le  \lim_{j\to \infty} f_1(x^{k_j}) +  \B_{\psi_k}(x^{k_j}, x^{k_j-1})  - \langle x^{k_j}, \xi^{k_j} \rangle + g_1^*(\xi^{k_j}) \\
	& \qquad\quad + f_2(y^{k_j})  + \B_{\varphi_k}(y^{k_j}, y^{k_j-1}) - \langle y^{k_j}, \eta^{k_j} \rangle + g_2^*(\eta^{k_j}) + h(x^{k_j},y^{k_j})\\
	&	\leq  \limsup_{j\to \infty} \left\{f_1(x^*) + h^+(x^*,y^{k_j-1}) - \langle x^*-x^{k_j},\uu^{k_j-1}\rangle + \B_{\psi_k}(x^*, x^{k_j-1}) \right.\\
	&\hskip1cm - \langle x^{k_j}, \xi^{k_j} \rangle + g_1^*(\xi^{k_j}) + f_2(y^*) + h^+(x^{k_j},y^*) - \langle y^*-y^{k_j}, \vv^{k_j-1}\rangle  \\
	&\hskip1cm \left.+ \B_{\varphi_k}(y^*, y^{k_j-1}) - \langle y^{k_j}, \eta^{k_j} \rangle + g_2^*(\eta^{k_j})- h^+(x^{k_j},y^{k_j-1})- h^-(x^{k_j},y^{k_j})\right\}\\
	&=  \limsup_{j\to \infty}\left\{ f_1(x^*) + h^+(x^*,y^{k_j-1}) - \langle x^*-x^{k_j}, \uu^{k_j-1}\rangle  + \B_{\psi_k}(x^*, x^{k_j-1}) \right. \\
	&\hskip1cm- g_1(x^{k_j-1})+ f_2(y^*) + h^+(x^{k_j},y^*) - \langle y^*-y^{k_j}, \vv^{k_j-1} \rangle + \B_{\varphi_k}(y^*, y^{k_j-1}) \\
	&\hskip1cm \left.- g_2(y^{k_j-1}) - h^+(x^{k_j},y^{k_j-1}) - h^-(x^{k_j},y^{k_j}) \right\}\\
	&= f_1(x^*) - g_1(x^*)  + f_2(y^*)  - g_2(y^*) + h(x^*,y^*) \equiv \Phi(x^*,y^*) \\
	&\le  \Psi(x^*,\xi^*,y^*,\eta^*) \equiv \Psi(\ww^*),
	\end{align*}
	where the second equality follows from \eqref{eq:conx}, the last equality holds by the continuity of $h(\cdot,\cdot)$ and $g_i(\cdot)$ $(i=1,2)$ and $\lim_{j \to \infty}\|\zz^{k_j}-\zz^{k_j-1}\| = 0$ (see Lemma \ref{lem:bound}), and the last inequality comes from \eqref{eq:PsiPhi}. In addition, we know from the low semicontinuity of $f_i(\cdot)$ ($i=1,2$) that the whole $\Psi(\ww)$ is also lower semicontinuous, which yields
	\begin{equation}
	\Psi(\ww^*)\equiv \Psi(x^*,\xi^*,y^*,\eta^*) \le \liminf_{j\to \infty}\Psi(x^{k_j},\xi^{k_j},y^{k_j},\eta^{k_j})\equiv  \liminf_{j\to \infty}\Psi(\ww^{k_j}).	
	\end{equation}
	Consequently, $ \lim_{j\to \infty}\Psi(\ww^{k_j})=\Psi(\ww^*)$. On the other hand, it follows from Lemmas \ref{lem:bound} and \ref{lem:disbound} that
	$$\lim_{j\to \infty} \dist(0,\partial \Psi(\ww^{k_j})) \le \lim_{j\to \infty} \tau \lno \zz^{k_j}-\zz^{k_j-1} \rno = 0.$$
	Hence, the closedness property of $\partial \Psi$ implies  that $0 \in \partial\Psi(\ww^*)$, which proves that $\ww^*$ is a critical point of $\Psi(\ww)$.
	
	Now, we turn our attention to Item (ii). According to Assumption~\ref{ass:2}, the sequence $\left\{\Psi(\ww^k)\right\}$ is bounded below. Besides, we see that the sequence $\left\{\Psi(\ww^k)\right\}$ is nonincreasing from \eqref{eq:dec_cond}. Thus, the limit $\Psi^\infty= \lim_{k\to \infty} \Psi(\ww^k)$ exists.
	
	Finally, we prove Item (iii). Taking any $\ww^*\in \mathbb{W}^*$ as an accumulation point of $\left\{\ww^k\right\}$, there exists a subsequence $\{\ww^{k_j}\}$ satisfying $\displaystyle \lim_{j \to \infty} \ww^{k_j} = \ww^*$. From Items (i) and (ii), we have
	$$\Psi(\ww^*) = \displaystyle \lim_{j \to \infty} \Psi(\ww^{k_j}) = \displaystyle \lim_{k\to \infty} \Psi(\ww^k) = \Psi^{\infty}.$$
	Thus, we conclude by the arbitrariness of $\ww^*$ that $\Psi(\ww) = \Psi^{\infty}$ for all $\ww\in\mathbb{W}^*$.
\end{proof}

With the above preparations, we are now proving that the sequence $\{\zz^k\}$ generated by Algorithm \ref{alg:proposed} globally converges to a critical point of \eqref{eq:problem}.
\begin{theorem}\label{thm1}
	Suppose that $\Psi(\ww)$ defined in~\eqref{eq:auxfun} is a K{\L} function such that Assumptions \ref{ass:1}--\ref{ass:3} hold. Let $\left\{\zz^k := (x^k,y^k) \right\}$ be a sequence generated by Algorithm \ref{alg:proposed} which is assumed to be bounded. Then, the following assertions hold.
	\begin{itemize}
		\item[\rm (i)] The sequence $\{\zz^k\}$ has finite length, i.e.,
		\begin{equation*}
		\sum_{k=1}^{\infty} \lno \zz^{k+1} - \zz^k \rno < \infty.
		\end{equation*}
		\item[\rm (ii)] The sequence $\{\zz^k\}$ converges to a critical point $\zz^* = (x^*,y^*)$ of $\Phi(x,y)$.
	\end{itemize} 
\end{theorem}

\begin{proof}
	Since the auxiliary function $\Psi(\ww)$ is a K{\L} function, we can apply Lemma \ref{lem:UKL} with setting $\Gamma=\mathbb{W}^*$ defined in Lemma \ref{prop:2} to our problem. Then, there exist $\varepsilon>0$, $\zeta>0$ and a continuous concave function $\phi(\cdot) \in \varUpsilon_\zeta$ such that, for any $\ww\in \left\{ \ww |~\dist (\ww,\mathbb{W}^*) <\varepsilon \right\} \cap \left\{ \ww|~\Psi^{\infty} < \Psi(\ww) < \Psi^{\infty} + \zeta\right\}$, we have
	\begin{equation*}
	\phi^\prime(\Psi(\ww)-\Psi^{\infty}) \cdot \dist(0,\partial \Psi(\ww)) \ge 1,
	\end{equation*}
	where $\Psi^{\infty} = \lim_{k\to \infty} \Psi(\ww^k)$. Hence there exists $\bar{N} > 0$ such that 
	\begin{equation}\label{th1:eq1}
	\phi^\prime(\Psi(\ww^k)-\Psi^{\infty}) \cdot \dist(0,\partial \Psi(\ww^k)) \ge 1,\;\quad \forall k \ge \bar{N}.
	\end{equation}	
	Using the concavity of $\phi(\cdot)$, we have 
	\begin{align}\label{lem:ieq}
	& \left[\phi\left(\Psi(\ww^k)-\Psi^{\infty}\right) - \phi\left(\Psi(\ww^{k+1})-\Psi^{\infty}\right)\right] \cdot \dist\left(0,\partial \Psi(\ww^k)\right) \nonumber \\
	& \ge \left[\Psi(\ww^k)- \Psi(\ww^{k+1})\right] \cdot \phi^\prime\left(\Psi(\ww^k)-\Psi^{\infty}\right) \cdot \dist\left(0,\partial \Psi(\ww^k)\right) \nonumber \\
	& \ge \Psi(\ww^k)-\Psi(\ww^{k+1}).
	\end{align}
	Defining $\mathscr{D}_k^\phi := \phi\left( \Psi( \ww^k)-\Psi^{\infty}\right) - \phi\left( \Psi( \ww^{k+1}) -\Psi^{\infty}\right)$, it then follows from Definition \ref{def:desf} that $\mathscr{D}_k^\phi $ is positive. Therefore, by recalling the sufficient decrease condition~\eqref{eq:dec_cond}, we have 
	\begin{align*}
	\lno \zz^{k+1}-\zz^k \rno^2 
	\le &\;\frac{2}{\gamma_k}\left(\Psi(\ww^k)-\Psi(\ww^{k+1})\right) \\
	\le &\; \frac{2}{\mu} \mathscr{D}_k^\phi \cdot \dist\left(0,\partial \Psi(\ww^k)\right) \\
	\le &\; \frac{2\tau}{\mu} \mathscr{D}_k^\phi \lno \zz^{k+1} - \zz^k \rno ,
	\end{align*}	
	where the second inequality and the third one follow from \eqref{lem:ieq} and the relative error condition~\eqref{eq:err_cond}, respectively.
	Moreover, an application of $\sqrt{ab} \le \frac{a+b}{2}$ for $a,b\in \RR_+$ to the last term of the above inequality yields
	\begin{equation*} 
	\lno \zz^{k+1} - \zz^k \rno  \le \sqrt{\frac{2\tau}{\mu}\mathscr{D}_k^\phi \lno \zz^{k+1} - \zz^k \rno} \le \frac{\tau}{\mu} \mathscr{D}_k^\phi + \frac{1}{2}\lno \zz^{k+1} - \zz^k \rno , \quad \forall k\ge \bar{N}.
	\end{equation*}
	Rearranging terms of the above inequality, we further obtain that	
	\begin{equation}\label{eq:converge_core}
	\lno \zz^{k+1} - \zz^k \rno \le \frac{2\tau}{\mu} \mathscr{D}_k^\phi.
	\end{equation}
	Summing both sides of the~\eqref{eq:converge_core} from $k=\bar{N}$ to $\infty$ and noting that 
	$$\sum_{k=\bar{N}}^\infty \mathscr{D}_k^\phi \le \phi\left( \Psi\left( \ww^{\bar{N}}\right) - \Psi^{\infty} \right)< \infty,$$
	we immediately obtain
	\begin{equation*}
	\sum_{k=\bar{N}}^\infty \lno \zz^{k+1} - \zz^k \rno  < \infty,
	\end{equation*}
	which implies that the sequence $\{\zz^k=(x^k,y^k)\}$ generated by our UBAMA is convergent.
	
	Suppose that $\lim_{k\to \infty} (x^k,y^k) = (x^*,y^*)$. Next we turn to proving $0 \in \partial \Phi(x^*,y^*)$. From Lemma~\ref{prop:2}, we see that there exists a subsequence $\{\ww^{k_j}\}$ converging to some point in $\mathbb{W}^*$, denoted by $\ww^\star$. Writing down $0 \in \partial \Psi(\ww^\star)$ yields
	\begin{equation*}
	\left\{
	\begin{aligned}
	0 &\in \partial f_1(x^\star) - \xi^\star + \nabla_xh(x^\star,y^\star),\\
	x^\star &\in \partial g_1^*(\xi^\star),\\
	0 &\in \partial f_2(y^\star) - \eta^\star + \nabla_yh(x^\star,y^\star),\\
	y^\star &\in  \partial g_2^*(\eta^\star).
	\end{aligned}\right.
	\end{equation*}
	Invoking the continuity and convexity of $ g_i(\cdot)$ ($i=1,2$) and the Fenchel-Young inequality \eqref{eq:FY}, we can obtain
	\begin{equation*}
	\left\{
	\begin{aligned}
	& 0 \in \partial f_1(x^\star) + \nabla_x h(x^\star,y^\star) - \partial g_1(x^\star), \\
	& 0 \in \partial f_2(y^\star) + \nabla_y h(x^\star,y^\star) - \partial g_2(y^\star).
	\end{aligned}\right.
	\end{equation*}
	Note that $(x^\star,y^\star) = \displaystyle\lim_{j\to \infty} (x^{k_j},y^{k_j}) = \lim_{k\to \infty} (x^k,y^k) = (x^*,y^*)$, we conclude that $0 \in \partial \Phi(x^*,y^*)$. This completes the proof.
\end{proof}

	To end this section, we prove the local convergence rates for our Algorithm \ref{alg:proposed} as proved in the seminal work \cite{AB09,ABRS10}. It is worth pointing out that the proof is based on the K{\L} assumption of the surrogate objective function $\Psi(\ww)$ defined in \eqref{eq:auxfun}, while such an assumption is imposed on the objective function for the novel PALM algorithm. Here, we refer the reader to \cite{LPT19}  for showing the relationships between various K{\L} assumptions, when dealing with only one block DC programming.
	
	\begin{theorem}
		Suppose that Assumptions \ref{ass:1}--\ref{ass:3} hold. Let $\left\{\zz^k := (x^k,y^k) \right\}$ be a sequence generated by Algorithm \ref{alg:proposed} which is assumed to be bounded. Suppose that $\Psi(\ww)$ is a K{\L} function with exponent $\vartheta \in [0,1)$, where the corresponding desingularizing function takes the form $\phi(s) = as^{1-\vartheta}$ with $a>0$. Then, the following assertions hold.
		\begin{itemize}
			\item[\rm (i)] If $\vartheta = 0$, then $\{\zz^k\}$ converges in finite iterations.
			\item[\rm (ii)] If $\vartheta \in (0, 1/2]$, then there exists $\hat{c}>0$ and $q\in [0,1)$ such that $\|\zz^k - \zz^*\| \le \hat{c} q^k$.
			\item[\rm (iii)] If $\vartheta \in (1/2, 1)$, then there exists $\hat{c}>0$ such that $\|\zz^k - \zz^*\| \le \hat{c} k^{-\frac{1-\vartheta}{2\vartheta-1}}$.
		\end{itemize} 
	\end{theorem}
	
	\begin{proof}
		By the condition of this theorem, $\Psi(\ww)$ given in \eqref{eq:auxfun} is assumed to be a K{\L} function with exponent $\vartheta$, then there exists some desingularizing function $\phi(s) = as^{1-\vartheta}$ satisfying the following K{\L} inequality (see \eqref{th1:eq1}):
		\begin{equation}\label{th2:ie1}
		\phi^\prime(\Psi(\ww^k)-\Psi^{\infty}) \cdot \dist(0,\partial \Psi(\ww^k)) \ge 1,\;\quad \forall k \ge \bar{N}.
		\end{equation}
		Invoking the definition of $\phi(\cdot)$ and Lemma \ref{lem:disbound}, it follows from \eqref{th2:ie1} that
		\begin{equation}\label{th2:ie2} 
		\frac{a(1-\vartheta)\tau \|\zz^k - \zz^{k-1}\|}{\left(\Psi(\ww^k)-\Psi^\infty\right)^\vartheta} \ge 1,\quad  \; \forall k \ge \bar{N}.
		\end{equation}
		Besides, it follows from \eqref{ineq:lem1} and $\mu= \inf_k \{\gamma_k\} > 0$ that
		\begin{equation}\label{th2:ie3} 
		\frac{\mu}{2}\lno \zz^k-\zz^{k-1}\rno^2 \le \Psi(\ww^{k-1}) - \Psi(\ww^k).
		\end{equation}
		Using the convexity of $-\phi(\cdot)$ gives
		$$
		-\phi(\Psi(\ww^k)-\Psi^\infty) \ge - \phi(\Psi(\ww^{k-1})-\Psi^\infty)  -\phi^\prime (\Psi(\ww^{k-1})-\Psi^\infty) \left(\Psi(\ww^k) - \Psi(\ww^{k-1})\right),
		$$
		which, by using the definition of $\phi(s)= as^{1-\vartheta}$, reads as
		\begin{align}\label{th2:ie4}
		&	 \left(\Psi(\ww^{k-1})-\Psi^\infty\right)^{1-\vartheta} - \left(\Psi(\ww^k)-\Psi^\infty\right)^{1-\vartheta} \nonumber\\
		&	\;\ge  (1-\vartheta)\left(\Psi(\ww^{k-1})-\Psi^\infty\right)^{-\vartheta}\left(\Psi(\ww^{k-1}) - \Psi(\ww^k)\right) \nonumber\\
		&\;	\ge  \frac{\mu(1-\vartheta)}{2}\lno \zz^k-\zz^{k-1}\rno^2 \left(\Psi(\ww^{k-1})-\Psi^\infty\right)^{-\vartheta} \nonumber\\
		&	\;\ge  \frac{\mu}{a \tau} \frac{\lno \zz^k - \zz^{k-1}\rno^2}{\lno \zz^{k-1} - \zz^{k-2}\rno},
		\end{align}
		where the second inequality and the last one follow from \eqref{th2:ie3} and \eqref{th2:ie2}, respectively.
		By the convergence of the sequence $\{\zz^k\}$, we assume that there exists some constant $r \in (0,1)$ such that $\|\zz^k - \zz^{k-1}\| \ge r\|\zz^{k-1} - \zz^{k-2}\|$ for sufficiently large $k$.
		Thus, it follows from \eqref{th2:ie4} that there exists some constant $\widehat{\alpha} = \frac{r\mu}{a\tau}>0$ such that
		\begin{equation}\label{ieq:rate-main}
		\frac{1}{\widehat{\alpha}}\lno \zz^k - \zz^{k-1} \rno \le \left(\Psi(\ww^{k-1})-\Psi^\infty\right)^{1-\vartheta} - \left(\Psi(\ww^k)-\Psi^\infty\right)^{1-\vartheta}.
		\end{equation}
		Let us define $\mathcal{S}_i = \sum_{k=i}^\infty \lno \zz^{k+1} - \zz^k \rno < \infty$, which is finite by Theorem \ref{thm1}. Accordingly, by using the triangular inequality and the fact that $\lim_{k\to \infty}\zz^k=\zz^*$, we can bound
		$$
		\begin{aligned}
		\|\zz^i - \zz^*\| & = \|(\zz^i - \zz^{i+1}) + (\zz^{i+1} - \zz^{i+2}) + \ldots + (\zz^{i+n} - \zz^*) \| \\
		& \le \sum_{j=1}^n \|\zz^{i+j} - \zz^{i+j-1}\| + \|\zz^{i+n} - \zz^*\| \\
		& \le p\mathcal{S}_i, \; \forall i \ge 0,
		\end{aligned}
		$$
		by some constant $p \ge 1$. It is therefore sufficient to establish the estimations appearing in (ii) and (iii) via $\mathcal{S}_i$. Without loss of generality, we assume  $\mathcal{S}_i>0$ and $\Psi(\ww^i) - \Psi^\infty \ge 0$ for all $i\ge 0$ due to the nonincreasing property of $\{\Psi(\ww^k)\}$. Hence, by summing inequality \eqref{ieq:rate-main} from $k=i+1$ to $\infty$, together with the fact that $\lno \zz^i - \zz^{i-1}\rno = \mathcal{S}_{i-1} - \mathcal{S}_i$, we immediately obtain
		\begin{align} \label{ieq:sumdz}
		\mathcal{S}_i \le  \widehat{\alpha} \left(\Psi(\ww^i)-\Psi^\infty\right)^{1-\vartheta} 
		\le  \widehat{\alpha} \left( a(1-\vartheta)\tau \lno \zz^i - \zz^{i-1}\rno \right)^{\frac{1-\vartheta}{\vartheta}}
		=  \widetilde{\alpha} \left(\mathcal{S}_{i-1} - \mathcal{S}_i\right)^{\frac{1-\vartheta}{\vartheta}}, 
		\end{align}
		where the second inequality comes from \eqref{th2:ie1} and $\widetilde{\alpha} = \widehat{\alpha} \left(a(1-\vartheta)\tau\right)^{\frac{1-\vartheta}{\vartheta}}$.
		
		Now, we are at the stage of proving the three assertions of this theorem.
		\begin{itemize}
			\item Case $\vartheta = 0$. For sufficiently large $k \in \mathcal{E} := \{k \in \NN : \zz^{k+1} \ne \zz^k\}$ we have $\lno \zz^{k+1} - \zz^k \rno \ge \varpi > 0$. It follows from \eqref{eq:dec_cond} that
			$$			\Psi(\ww^{k+1}) \leq \Psi(\ww^k) - \frac{\mu \varpi^2}{2}.		$$
			The convergence of $\Psi(\ww^k)$ implies that $\mathcal{E}$ is finite.
			\item Case $\vartheta \in (0, 1/2]$. For sufficiently large $k$ (say, $k\ge \tilde{k}$), it holds that $(\mathcal{S}_{k-1} - \mathcal{S}_k)^{\frac{1-\vartheta}{\vartheta}} \le \mathcal{S}_{k-1} - \mathcal{S}_k$,
			which, together with \eqref{ieq:sumdz}, gives
			$$\mathcal{S}_k \le \frac{\widetilde{\alpha}}{1+\widetilde{\alpha}} \mathcal{S}_{k-1}.$$
			Furthermore, by induction on $k$, we can obtain
			$$\|\zz^k - \zz^*\| \le p\mathcal{S}_k \le p\left(\frac{\widetilde{\alpha}}{1+\widetilde{\alpha}}\right)^{k-\tilde{k}}\mathcal{S}_{\tilde{k}} = \hat{c} \left(\frac{\widetilde{\alpha}}{1+\widetilde{\alpha}}\right)^k,$$
			with some $\hat{c}>0$. Hence, setting $q = {\widetilde{\alpha}}/{(1+\widetilde{\alpha})}$ arrives at the second assertion. 
			\item Case $\vartheta \in (1/2,1)$. 
			By defining a function ${\bm F}(s) = s^{-\frac{\vartheta}{1-\vartheta}}$ and letting $\widehat{\upsilon} \in (1,\infty)$, we first assume ${\bm F}(\mathcal{S}_i) \le \widehat{\upsilon} {\bm F}(\mathcal{S}_{i-1})$ and take $k\ge \widehat{k}$. Rewriting inequality \eqref{ieq:sumdz}, i.e., $\mathcal{S}_i \le \widetilde{\alpha} (\mathcal{S}_{i-1} - \mathcal{S}_i)^{\frac{1-\vartheta}{\vartheta}}$, arrives at 
			\begin{align*}
			1 \le  \frac{\widetilde{\alpha}^{\frac{\vartheta}{1-\vartheta}}(\mathcal{S}_{i-1} - \mathcal{S}_i)}{\mathcal{S}_i^{\frac{\vartheta}{1-\vartheta}}} 
			& =\widetilde{\alpha}^{\frac{\vartheta}{1-\vartheta}}(\mathcal{S}_{i-1} - \mathcal{S}_i){\bm F}(\mathcal{S}_i)  \nonumber\\
			&\le  \widehat{\upsilon}\widetilde{\alpha}^{\frac{\vartheta}{1-\vartheta}}(\mathcal{S}_{i-1} - \mathcal{S}_i){\bm F}(\mathcal{S}_{i-1}) \nonumber\\
			&	\le  \widehat{\upsilon}\widetilde{\alpha}^{\frac{\vartheta}{1-\vartheta}}\int_{\mathcal{S}_i}^{\mathcal{S}_{i-1}} {\bm F}(s) \mathrm{d}s \nonumber\\
			&\le  \widehat{\upsilon}\widetilde{\alpha}^{\frac{\vartheta}{1-\vartheta}}\frac{1-\vartheta}{1-2\vartheta}\int_{\mathcal{S}_i}^{\mathcal{S}_{i-1}} {\bm F}(s) \mathrm{d}s\\
			&\le  \widehat{\upsilon}\widetilde{\alpha}^{\frac{\vartheta}{1-\vartheta}}\frac{1-\vartheta}{1-2\vartheta}\left[\mathcal{S}_{i-1}^{\frac{1-2\vartheta}{1-\vartheta}} - \mathcal{S}_i^{\frac{1-2\vartheta}{1-\vartheta}}\right].
			\end{align*}
			Note that $\frac{1-2\vartheta}{1-\vartheta} < 0$ for $\vartheta \in (1/2, 1)$. Then, we have
			\begin{equation} \label{eq:h-case1}
			0 < -\frac{1-2\vartheta}{(1-\vartheta)\widehat{\upsilon}\widetilde{\alpha}^{\frac{\vartheta}{1-\vartheta}}} \le \mathcal{S}_i^{\frac{1-2\vartheta}{1-\vartheta}} - \mathcal{S}_{i-1}^{\frac{1-2\vartheta}{1-\vartheta}}.
			\end{equation}
			Now assume ${\bm F}(\mathcal{S}_i) > \widehat{\upsilon}{\bm F}(\mathcal{S}_{i-1})$ and set $q = \widehat{\upsilon}^{-\frac{1-\vartheta}{\vartheta}} \in (0,1)$, then $\mathcal{S}_i \le q \mathcal{S}_{i-1}$. Furthermore, $\mathcal{S}_i^{\frac{1-2\vartheta}{1-\vartheta}} \ge q^{\frac{1-2\vartheta}{1-\vartheta}} \mathcal{S}_{i-1}^{\frac{1-2\vartheta}{1-\vartheta}}$ and
			\begin{equation} \label{eq:h-case2}
			0 < (q^{\frac{1-2\vartheta}{1-\vartheta}} -1) \mathcal{S}_{i-1}^{\frac{1-2\vartheta}{1-\vartheta}} \le \mathcal{S}_i^{\frac{1-2\vartheta}{1-\vartheta}} - \mathcal{S}_{i-1}^{\frac{1-2\vartheta}{1-\vartheta}}.
			\end{equation}
			Both inequalities \eqref{eq:h-case1} and \eqref{eq:h-case2} show that there always exists some constant $\nu>0$ such that
			$$	\mathcal{S}_i^{\frac{1-2\vartheta}{1-\vartheta}} - \mathcal{S}_{i-1}^{\frac{1-2\vartheta}{1-\vartheta}} \ge \nu.$$
			Summing this inequality from $i = \widehat{k}+1$ to $i=k$, we obtain $\mathcal{S}_k^{\frac{1-2\vartheta}{1-\vartheta}} - \mathcal{S}_{\widehat{k}}^{\frac{1-2\vartheta}{1-\vartheta}} \ge \nu (k-\widehat{k})$, which can be reformulated as
			$$
			\|\zz^{k} - \zz^*\| \le p\mathcal{S}_k \le p\left[ \mathcal{S}_{\widehat{k}}^{\frac{1-2\vartheta}{1-\vartheta}} + \nu(k-\widehat{k}) \right]^{\frac{1-\vartheta}{1-2\vartheta}} \le \hat{c} k^{-\frac{1-\vartheta}{2\vartheta-1}}
			$$
			for some $\hat{c}>0$. The proof is complete.
		\end{itemize}
	\end{proof}

The K{\L} exponent plays an important role in analyzing the convergence rate of first-order optimization methods. However, the estimation of the K{\L} exponent is not an easy task in general. Here, we refer the reader to \cite{BNPS17,LP18,WPB21,YLP22} for some novel results on the K{\L} exponent of sparse optimization problems.

\section{Numerical experiments} \label{Sec:Exp}
In this section, we aim to show that our Algorithm \ref{alg:proposed} (denoted by `UBAMA') is a customized solver for generalized DC programming \eqref{eq:problem} in the sense that our UBAMA enjoys easier subproblems, thereby taking less computing time than some existing state-of-the-art nonconvex optimization methods. Accordingly, we first modify some existing imaging optimization models so that they fall into the form of \eqref{eq:problem}. Then, we conduct the numerical performance of our UBAMA on some imaging datasets. As tested in the nonconvex literature, here we will not check the conditions required in theoretical analysis for the coming models, since the quality of these obtained solutions can be seen from the recovered images. All numerical experiments have been implemented in {\sc Matlab} 2022a and been conducted on a ThinkPad T470 laptop computer with Intel(R) Core(TM) i5 CPU 2.40 GHz and 8G memory.

\subsection{Image reconstruction}\label{subsec:image}
Image deconvolution and inpainting are prototypical image reconstruction problems. In this subsection, we consider the task of reconstructing an image from the one which is convoluted (or blurred) and suffered some noise and loss of information, i.e., reconstructing image $x\in \RR^n$ from the following system
\begin{equation}\label{eq:imagedi}
b=SKx + \varepsilon,
\end{equation}
where $b$ is a corrupted image with Gaussian white noise $\varepsilon \sim \mathcal{N}(0,\delta^2)$, $S: \RR^n \to \RR^m$ and $K : \RR^n \to \RR^n$ are the down-sampling and spatially invariant convolution operators, respectively. In general, such a problem is ill-posed and more difficult than the pure image deconvolution and pure image inpainting. As we know, one of the most popular ways to recover $x$ is reformulating \eqref{eq:imagedi} as a least square problem equipped with a powerful Total Variation (TV) regularization term \cite{ROF92}.  In 2015, Lou et al. \cite{LZOX15} introduced a  weighted difference of anisotropic and isotropic TV model, and showed numerically that such a model is powerful for image denoising, image deblurring, and magnetic resonance imaging reconstruction. Here, we follow this idea to tackle \eqref{eq:imagedi}, and the corresponding model is expressed as follows
\begin{equation}\label{eq:dcimagedi}
\min_{x} \; \frac{1}{2} \lno SKx - b \rno_2^2 + \tau \left(\lno \D x \rno_1 - \alpha \lno \D x \rno_{2,1}\right),
\end{equation}
where $\|\D x \|_1=\| \D_1 x \|_1+\|\D_2 x\|_1$ and $ \|\D x\|_{2,1}=\lno \sqrt{| \D_1 x|^2 + |\D_2 x|^2}\rno_1$ with $\D_1,\D_2$ being the horizontal and vertical partial derivative operators, respectively (see \cite{LZOX15} for more details); $\tau>0$ and $\alpha \in [0,1]$ are trade-off parameters to balance the regularization term. To simplify the DC part, we first introduce an auxiliary variable $y = \D x$ to extract $\D x$ from the nonsmooth functions. Then, we follow the spirit of penalty method to reformulate \eqref{eq:dcimagedi} as
\begin{equation} \label{eq:super-resolution}
\min_{x,y} \; \frac{1}{2}\lno SKx - b\rno_2^2 + \tau \left(\lno y \rno_1 - \alpha \lno y\rno_{2,1} \right) + \frac{\beta}{2}\lno \D x - y \rno^2,
\end{equation}
where $\beta>0$ is a penalty parameter. Clearly, model \eqref{eq:super-resolution} is a special case of \eqref{eq:problem} by setting
\begin{align*}
&f_1(x) = \frac{1}{2} \lno SKx -b \rno_2^2, \quad g_1(x)=0,\quad f_2(y) = \tau \lno y \rno_1, \quad g_2(y)=\tau\alpha \lno y \rno_{2,1} \\
&h^+(x,y) = \frac{\beta}{2}\lno \D x-y \rno^2,\quad h^-(x,y)=0,
\end{align*}
respectively. Therefore, our algorithm is able to find a solution of \eqref{eq:super-resolution}. On the other hand, when the DC regularized part is regarded as a general nonconvex function, model \eqref{eq:super-resolution} falls into the case discussed in \cite{BST14}. So, we here employ the PALM algorithm (see \cite{BST14} and also \eqref{eq:PALM}) to solve \eqref{eq:super-resolution}, which will be compared with our UBAMA for the purpose of demonstrating the efficiency of our approach.

Hereafter, we present the details of implementing our UBAMA to solve \eqref{eq:super-resolution}. First, we take the Bregman kernel functions $\psi_k(\cdot) = \frac{1}{2}\|\cdot\|_{M_k}^2$ and $\varphi_k(\cdot)= \frac{\nu_k}{2}\|\cdot\|^2$ for the Bregman proximal terms $\B_{\psi_k}$ and $\B_{\varphi_k}$, where $M_k$ is specified as $M_k=\mu_k K^\top K - K^\top S^\top SK$ with an appropriate $\mu_k>0$ such that $M_k$ is positive definite. To make our UBAMA for \eqref{eq:super-resolution} readable, we then list the algorithmic details in Algorithm \ref{alg:ubama-one}, which is equipped with the following stopping criterion:
\begin{equation}
{\rm Tol}:= \max\left\{ \frac{\lno x^{k+1} - x^k \rno}{\max\{1,\lno x^k \rno\}}, \frac{\lno y^{k+1} - y^k \rno}{\max\{1,\lno y^k \rno\}} \right\} < 10^{-4}.
\end{equation}
\begin{algorithm}[!htb]
	\caption{Details of applying UBAMA  to \eqref{eq:super-resolution}.}\label{alg:ubama-one}
	\begin{algorithmic}[1]
		\STATE Select $\mu_k$ satisfying $M_k = \mu_k K^\top K -  K^\top S^\top SK  \succ 0$, $\nu_k>0$ and starting point $(x^0,y^0)$.
		\FOR{$k=0,1,2,\ldots$}
		\STATE Compute $x^{k+1}$ via
		\begin{align*}
		x^{k+1} & = \arg\min_x  \left\{\frac{1}{2} \lno SKx - b\rno^2 + \frac{\beta}{2}\lno \D x - y^k\rno^2 + \frac{1}{2}\lno x-x^k\rno_{M_k}^2\right\}\\
		& = \left( \beta \D^\top \D + \mu_k K^\top K \right)^{-1} \left[  K^\top S^\top b + \beta \D^\top y^k +  \left( \mu_k K^\top K - K^\top S^\top SK \right) x^k\right].
		\end{align*}		
		\STATE Take $\eta^{k+1}\in \partial \lno y^k\rno_{2,1}$ via
		\begin{equation*}
		[\eta^{k+1}]_{i,j} =
		\begin{cases} \frac{1}{\sqrt{[y^k]_i^2 + [y^k]_j^2}}\left([y^k]_i, [y^k]_j\right),\; &\text{if $[y^k]_i^2 + [y^k]_j^2 \ne 0$}, \\							(0,0),\; &\text{otherwise},
		\end{cases}\;\;\;
		i,j= 1,2,\ldots,n.
		\end{equation*}
		\STATE Update $y^{k+1}$ via
		\begin{align}\label{eq:ubama-y-sub-tv}
		y^{k+1} & =\arg\min_y \left\{\tau \lno  y\rno_1 - \tau\alpha \left \langle \eta^{k+1}, y\right\rangle + \frac{\beta}{2}\lno y-\D x^{k+1} \rno^2 + \frac{\nu_k}{2} \lno y - y^k\rno^2\right\}  \\
		& = \shrink \left( \frac{1}{\beta + \nu_k}\left(\beta \D x^{k+1} + \nu_k y^k+\tau\alpha \eta^{k+1}\right), \frac{\tau}{\beta+\nu_k} \right).\nonumber
		\end{align}
		\ENDFOR
	\end{algorithmic}
\end{algorithm}
We can see from Algorithm \ref{alg:ubama-one} that each subproblem of our UBAMA enjoys a closed-form solution.  It is noteworthy that there are some other choices on the Bregman kernel functions. For example, when taking $\psi_k(\cdot) = \frac{1}{2}\|\cdot\|_{M_k}^2$  with $M_k=c_k I - \beta\D^\top \D$,  the $x$-subproblem \eqref{algeq:update_x} of UBAMA is specified as 
\begin{align}\label{eq:xsub-palm}
x^{k+1} & = \arg\min_x \left\{ \frac{1}{2} \lno SKx - b\rno^2 + \frac{\beta}{2}\lno \D x - y^k\rno^2 + \frac{1}{2}\lno x-x^k\rno_{M_k}^2\right\}\\
& = \arg\min_x  \left\{\frac{1}{2} \lno SKx - b\rno^2 + \beta \left\langle \D^\top \left(\D x^k -  y^k\right), x-x^k\right\rangle + \frac{c_k}{2}\lno x-x^k\rno^2\right\} \nonumber\\
& = \left(  K^\top S^\top SK + c_k I\right)^{-1} \left(  K^\top S^\top b + c_k x^k - \beta \D^\top \left( \D x^k - y^k\right) \right),\nonumber
\end{align}
which precisely corresponds to the $x$-subproblem of  the PALM algorithm \eqref{eq:PALM}. Besides, the $y$-subproblem of the PALM algorithm for \eqref{eq:super-resolution} reads as
\begin{equation}\label{palm:ysub}
y^{k+1}=\arg\min \left\{\tau \left(\lno y \rno_1 - \alpha \lno y\rno_{2,1} \right) + \frac{\beta}{2}\lno \D x^{k+1} - y \rno^2+\frac{\nu_k}{2}\|y-y^k\|^2\right\},
\end{equation}
which in general requires an optimization solver to find an approximation instead of its accurate solution. In our experiments, by utilizing the simple DC structure, we employ one step approximation \eqref{eq:ubama-y-sub-tv} instead of calling an optimization solver to find an approximate solution of \eqref{palm:ysub} for the purpose of saving computing time, where $\nu_k$ is set as $\nu_k=d_k-\beta$. It is clear that the $x$-subproblem \eqref{eq:xsub-palm} of the PALM algorithm also enjoys a theoretical closed-form solution. However, the simultaneous appearance of $K$ and $S$ makes the $x$-subproblem extremely ill-posed so that solving it directly via the inverse formula will result in unstable solutions. Hence, we here compare our UBAMA (i.e., Algorithm \ref{alg:ubama-one}) with  the PALM algorithm (i.e., \eqref{eq:xsub-palm} and \eqref{eq:ubama-y-sub-tv}) for solving model \eqref{eq:super-resolution}, where the $x$-subproblem \eqref{eq:xsub-palm} is solved approximately by the well-developed Preconditioned Conjugate Gradient (PCG) method.
 We set the model parameter $\alpha$ as $\alpha = 0.1$. Besides, we consider two scenarios on the operator $K$, and set $(\tau,\beta) =(0.7 \delta, 50 \delta)$ and $(\tau,\beta) =(4\times 10^{-4}, 2\times 10^{-2})$ for the cases $K=I$ (i.e., image inpainting and denoising) and $K\neq I$ (i.e., image deconvolution and inpainting), respectively, where $\delta$ is the standard deviation of noise.

Now, we conduct numerical simulations of these algorithms on the first scenario: Image inpainting and denoising ($K=I$), where the images (i.e., {\sf Barbara}, {\sf Pepper}, {\sf Cameraman}, {\sf Roof}, and {\sf House}) are of size $256\times 256$ and are corrupted by adding Gaussian noise with different levels (i.e., $\delta=\{0.05,0.10,0.15,0.20\}$) and by dropping pixels with different masks (i.e., block-wise random, scratch, and text masks). In Figure \ref{fig:compare_images}, we summarize the observed images and recovered images by UBAMA and the PALM algorithm. It can be seen that both UBAMA and PALM have almost the same recovery quality. Furthermore, we report the number of iterations (Iter.), computing time in seconds (Time), the Signal-to-Noise Ratio (SNR) defined by 
$$\text{SNR}(x)=20\log_{10}\frac{\|x\|}{\|x^\star -x\|},$$
and the {\it structural similarity} (SSIM\footnote{A {\sc Matlab} package for	SSIM: https://ece.uwaterloo.ca/$\sim$z70wang/research/ssim/.}) index (see  \cite{WBSS04} and also \cite{LZOX15}), which are used to measure the quality of an image, where $x^\star$ and $x$ represent the ground truth image and a reconstructed image, respectively.
In Table \ref{tab:denoising_compare_indices}, we can see that our UBAMA runs a little faster than the PALM algorithm, which demonstrate that our UBAMA possessing easy subproblems can speed up the procedure of solving \eqref{eq:super-resolution} duo to its easy subproblems. As a visual complement for the convergence behavior, the convergence curves of objective values and SNR values with respect to iterations are demonstrated in Figure \ref{fig:compare_curves}, which also shows that our UBAMA converges faster than the PALM algorithm.

\begin{figure}[!tb]
	\begin{center}
		\includegraphics[width=\textwidth]{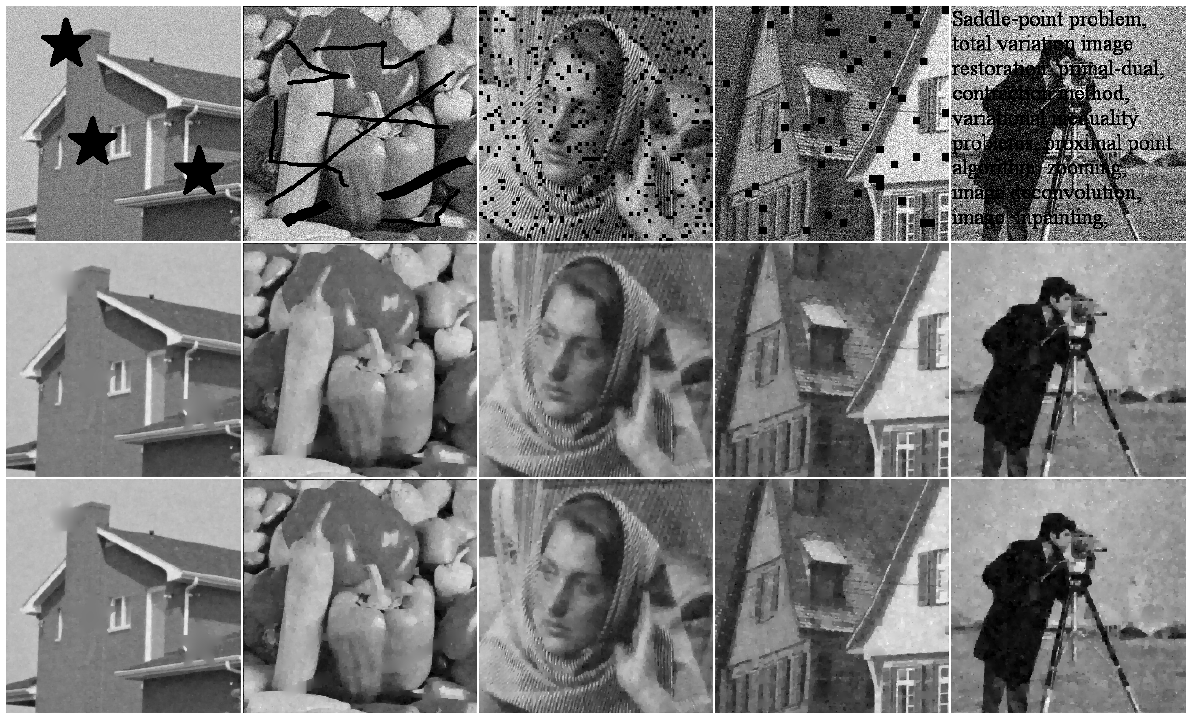}
	\end{center}	
	\caption{Results for image inpainting and denoising. From left to right: scenarios with noise levels $\delta = 0.05, 0.10, 0.10, 0.15, 0.20$, respectively. From top to bottom: the observed images, the reconstructed images by UBAMA and the PALM algorithm, respectively. }
	\label{fig:compare_images}
\end{figure}

\begin{figure}[!tb]
	\begin{center}
		 \includegraphics[width = 0.32\textwidth]{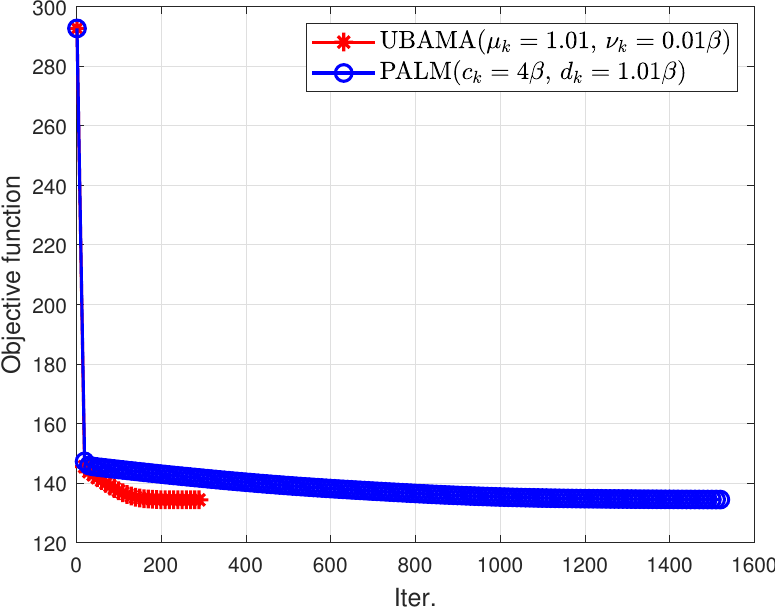}		
		\includegraphics[width = 0.32\textwidth]{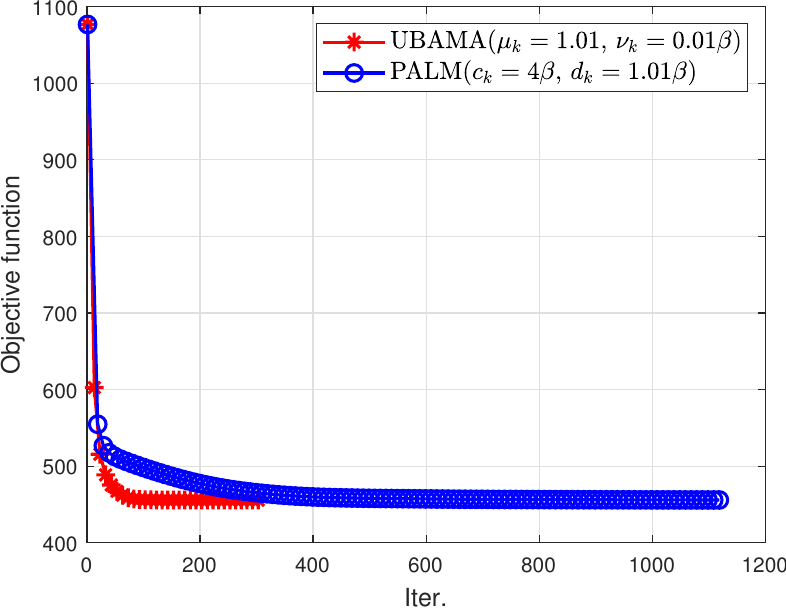}
		\includegraphics[width = 0.32\textwidth]{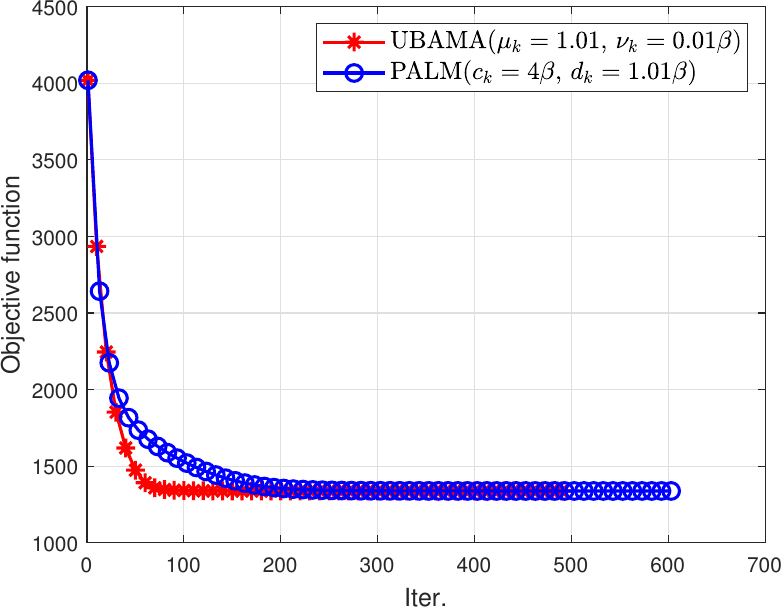}\\
		 \includegraphics[width = 0.32\textwidth]{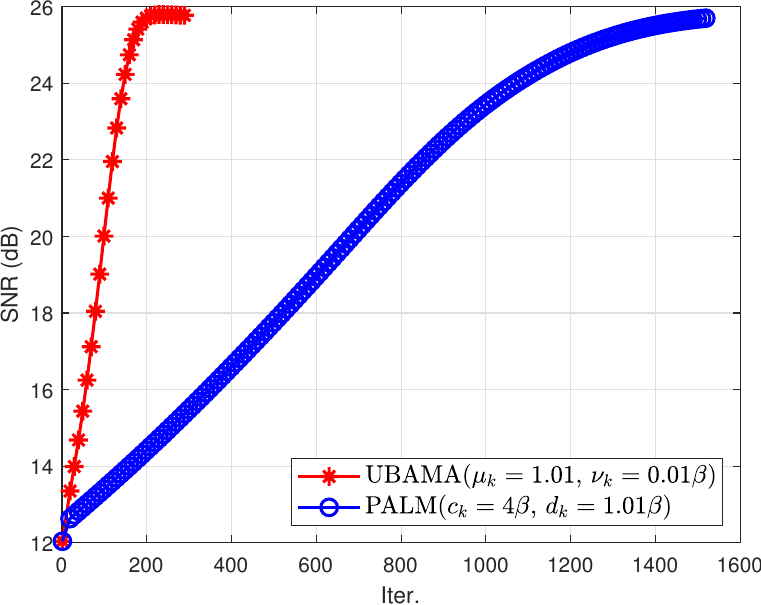}		
		\includegraphics[width = 0.32\textwidth]{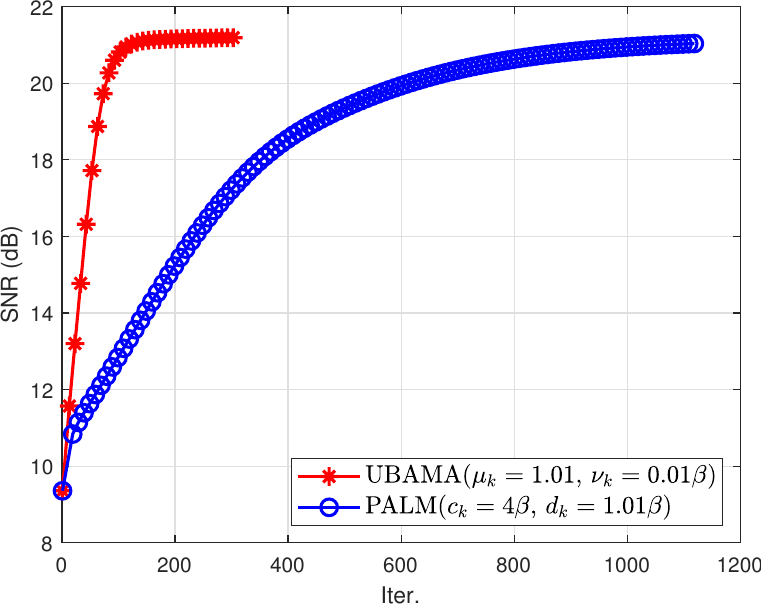}
		\includegraphics[width = 0.32\textwidth]{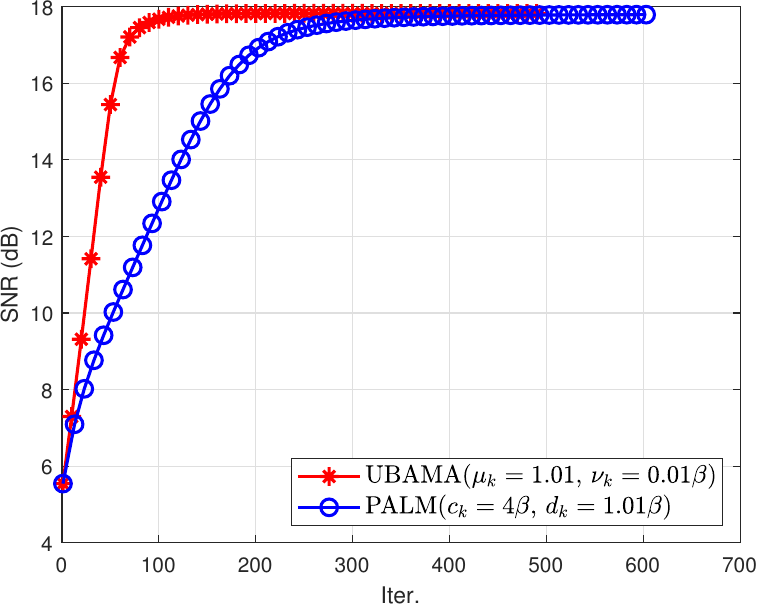}
	\end{center}
	\caption{Evolution of objective and SNR values with respect to iterations for image inpainting and denoising. From left to right: {\sf House}, {\sf Pepper} and {\sf Cameraman}.}
	\label{fig:compare_curves}
\end{figure}

\begin{table}[!tb]
	\caption{Numerical comparisons for image inpainting and denoising.}
	\label{tab:denoising_compare_indices}
	\centering
	\begin{tabular*}{\textwidth}{@{\extracolsep{\fill}}ccccc}
		\toprule
		Methods    && UBAMA && PALM \\ 
		\cline{1-1} \cline{3-3}  \cline{5-5} 
		Test Image && SNR / SSIM / Iter. / Time(s) && SNR / SSIM / Iter. / Time(s) \\ \midrule 
		{\sf House}      && 25.778 / 0.8483 / 289 / 1.82 && 25.702 / 0.8474 / 1519 / 7.51 \\
		{\sf Pepper}     && 21.191 / 0.7959 / 303 / 2.03 && 21.040 / 0.7946 / 1119 / 5.61 \\ 
		{\sf Barbara}    && 17.686 / 0.6920 / 258 / 1.65 && 17.684 / 0.6919 / 493 / 2.72 \\ 
		{\sf Roof}       && 17.302 / 0.6217 / 260 / 1.52 && 17.295 / 0.6215 / 1026 / 5.46 \\ 
		{\sf Cameraman}  && 17.823 / 0.6458 / 490 / 2.95 && 17.794 / 0.6456 / 603 / 2.98 \\ 
		\bottomrule
	\end{tabular*}
\end{table}

Below, we turn our attention to the second scenario: Image deblurring and inpainting (i.e., $K\neq I$), where the images are blurred by {\sc Matlab} script {\sf K=fspecial('disk',radius)} and some pixels are dropped in three ways as used in the above experiments. 
Also, some low-level Gaussian noise (i.e., $\delta=0.01$) is added in these corrupted images. In this part, we set the blur kernel size as $5\times 5$. As we have mentioned, it is not practical to solve the $x$-subproblem \eqref{eq:xsub-palm} directly for the case $K\neq I$. Hence, we employ the PCG method to find its approximate solution. In this situation, we set the tolerance of the optimization subroutine (i.e., PCG method) as $10^{-p}$ for the PALM algorithm, which means that the PCG method returns an approximate solution with precision $10^{-p}$. So, we are interested in the importance of the accuracy of solving subproblems for the numerical performance of the PALM algorithm. 
Here, we consider three different tolerances $\{10^{-3}, 10^{-4},10^{-5}\}$ (denoted by PALM(1e-3), PALM(1e-4), PALM(1e-5), respectively) for the optimization subroutine. The numerical results for this scenario are shown in Figures \ref{fig:blur_compare_images} and \ref{fig:blur_compare_curves}, and more detailed values are summarized in Table \ref{tab:deblurring_compare_indices}. 
It can be easily seen from the reported results that the UBAMA, PALM(1e-4), PALM(1e-5) achieve almost the same recovery quality for these images. However, we can see that when the PALM algorithm is equipped with a low-precision subroutine (i.e., low-accuracy solutions of the subproblems), the PALM algorithm does not work well in terms of image quality for some cases, e.g., the  image {\sf Pepper} with a scratch mask missing. 
The numerical results in Table \ref{tab:deblurring_compare_indices} demonstrate that our UBAMA runs much faster than the PALM algorithm, when PALM requires high-precision subproblems' solutions. These results further support that our structure-exploiting algorithm UBAMA is efficient and reliable for solving generalized DC programming \eqref{eq:problem}.

\begin{table}[!tb]
	\caption{Numerical comparisons for image deblurring, inpainting and denoising.}\label{tab:deblurring_compare_indices}
	\centering{\small
		\begin{tabular*}{\textwidth}{@{\extracolsep{\fill}}ccccccc}
			\toprule
			Test Image && {\sf Barbara} && {\sf Pepper }\\ 
				\cline{1-1}\cline{3-3}  \cline{5-5} 
			Method     && Obj. / SNR / SSIM / Iter. / Time && Obj. / SNR / SSIM / Iter. / Time \\ \midrule 
			UBAMA      && 4.202 / 19.240 / 0.7719 / 401 / 5.32  && 3.740 / 22.914 / 0.8196 / 3513 / 42.92 \\ 
			PALM(1e-3) && 4.266 / 19.419 / 0.7712 / 107 / 7.99  && 4.005 / 14.471 / 0.7412 / 156 / 12.31 \\ 
			PALM(1e-4) && 4.204 / 19.282 / 0.7727 / 235 / 25.20  && 3.739 / 23.073 / 0.8213 / 526 / 60.42 \\ 
			PALM(1e-5) && 4.202 / 19.243 / 0.7720 / 278 / 37.19  && 3.737 / 23.181 / 0.8225 / 507 / 87.12 \\ 
			\toprule
			Test Image &&{\sf House} && {\sf Cameraman} \\ 
				\cline{1-1}\cline{3-3}  \cline{5-5} 
			Method     && Obj. / SNR / SSIM / Iter. / Time && Obj. / SNR / SSIM / Iter. / Time \\ \midrule 
			UBAMA      && 2.734 / 26.621 / 0.8025 / 374 / 4.19 && 3.826 / 22.467 / 0.7900 / 311 / 3.63 \\ 
			PALM(1e-3) && 2.766 / 26.975 / 0.8028 / 133 / 9.06 && 3.946 / 21.778 / 0.7735 / 71 / 5.08 \\ 
			PALM(1e-4) && 2.736 / 26.685 / 0.8024 / 259 / 22.85 && 3.828 / 22.460 / 0.7894 / 240 / 26.05 \\ 
			PALM(1e-5) && 2.734 / 26.622 / 0.8025 / 309 / 34.42 && 3.826 / 22.465 / 0.7900 / 240 / 32.58 \\ 
			\bottomrule
	\end{tabular*}}
\end{table}

\begin{figure}[!tb]
	\begin{center}
		\includegraphics[width=0.8\textwidth]{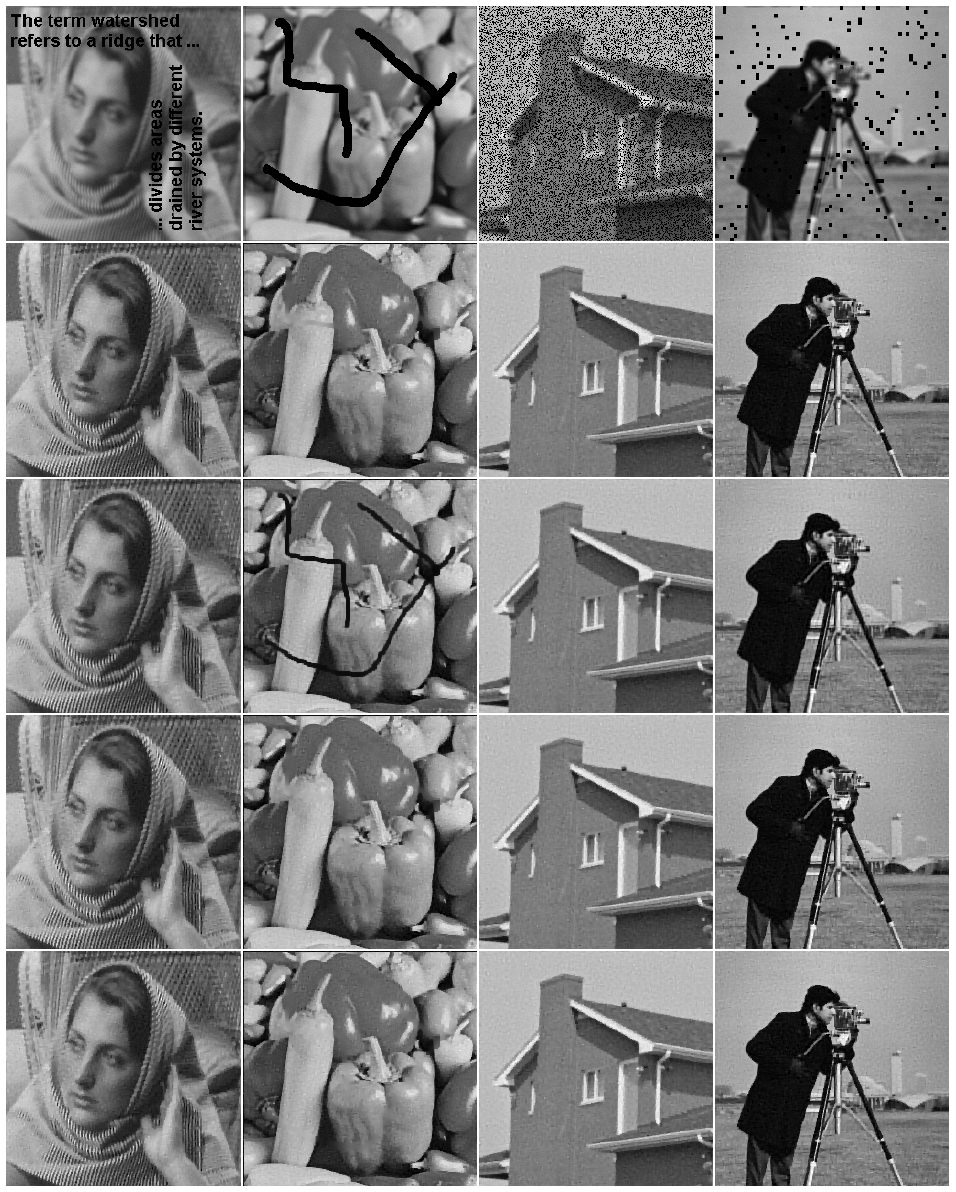}
	\end{center}
	\caption{Results for image deblurring, inpainting and denoising. Each observed image is blurred by Gaussian kernel of size $5\times 5$ with standard deviation $5$ and corrupted by adding Gaussian white noise with standard deviation $\delta = 0.01$. From top to bottom: the observed images, images restored by UBAMA, PALM(1e-3), PALM(1e-4), PALM(1e-5), respectively.}
	\label{fig:blur_compare_images}
\end{figure}

\begin{figure}[!tb]
	\begin{center}
		\includegraphics[width=0.32\textwidth]{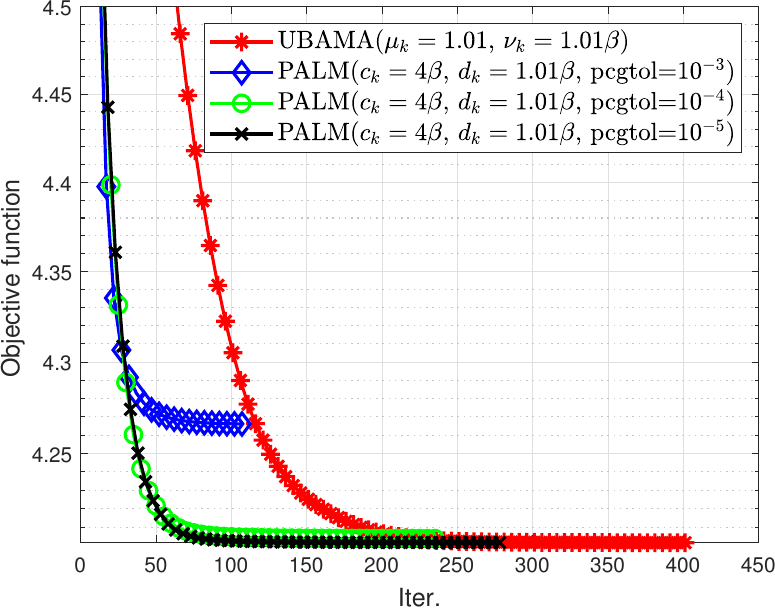}
		\includegraphics[width=0.32\textwidth]{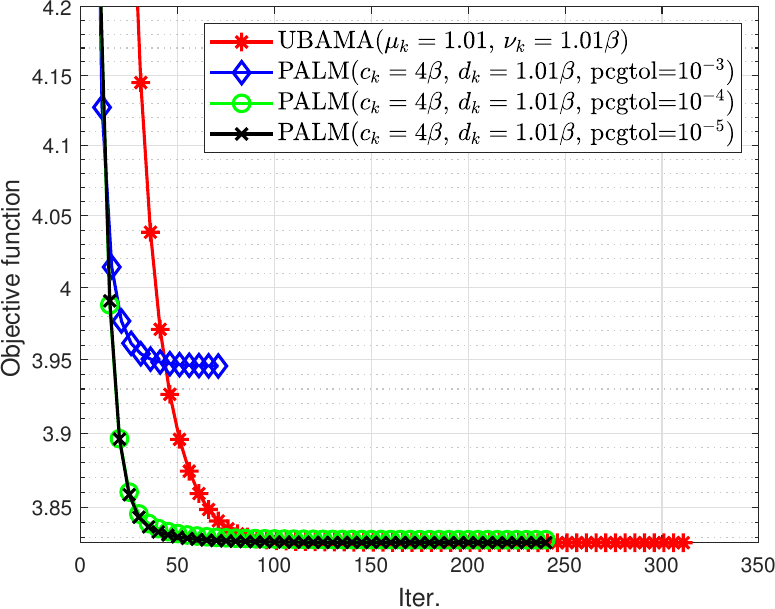}
		\includegraphics[width=0.32\textwidth]{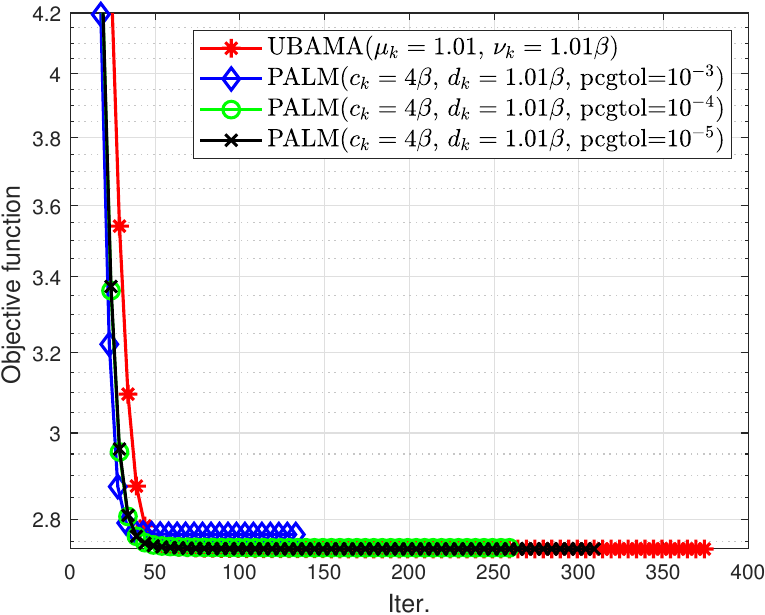}\\
		\includegraphics[width=0.32\textwidth]{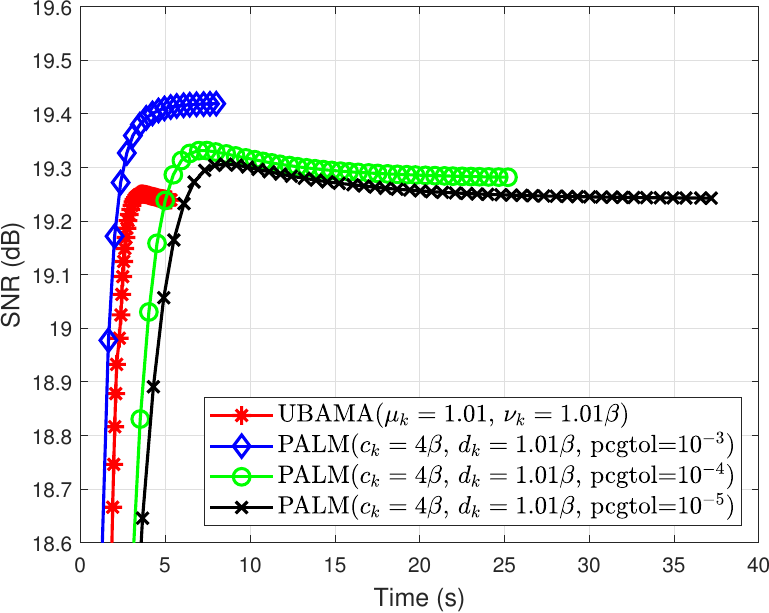}
		\includegraphics[width=0.32\textwidth]{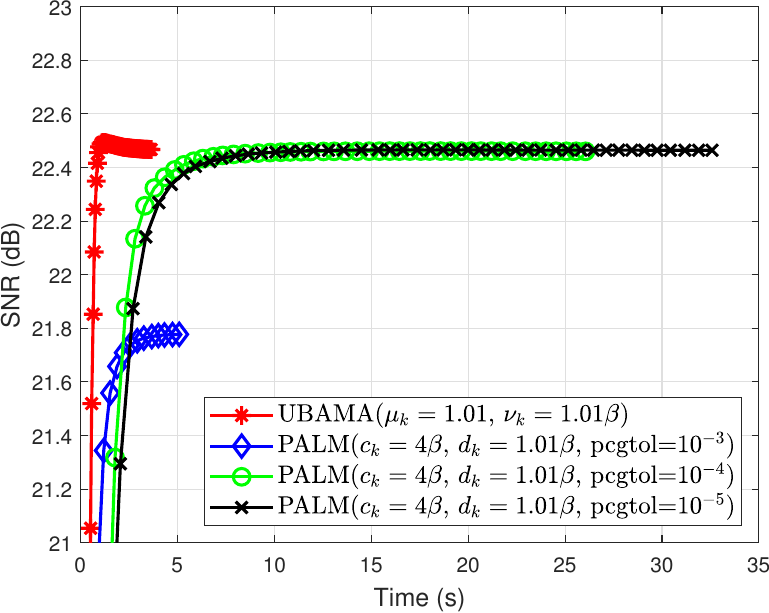}	
		\includegraphics[width=0.32\textwidth]{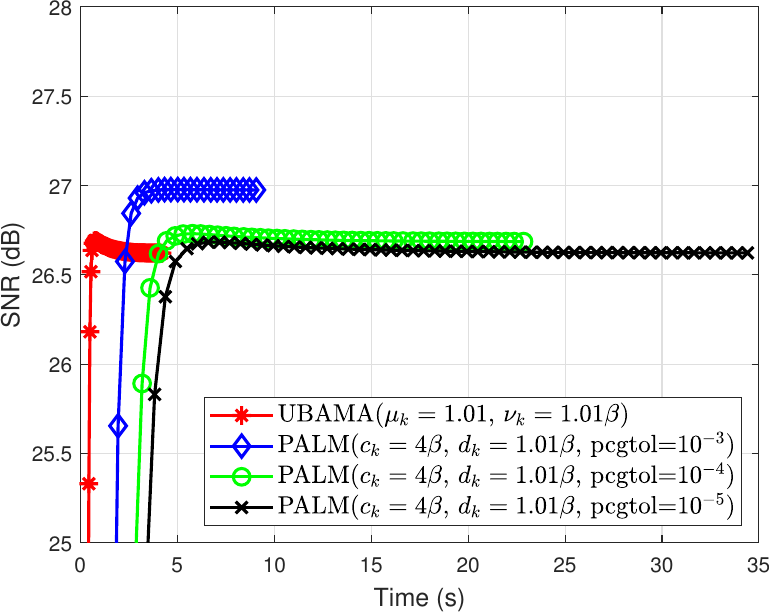}
	\end{center}
	\caption{Convergence curves for image deblurring, inpainting and denoising. From left to right: {\sf Barbara} and {\sf Cameraman} and {\sf House}.}
	\label{fig:blur_compare_curves}
\end{figure}

\subsection{Robust principal component analysis}\label{subsec:rpca}
The model \eqref{eq:super-resolution} in Section \ref{subsec:image} has only one DC part. Now, in this subsection, we will consider a more general case where there exist two separable DC parts in the objective function. More concretely, we consider the well-known Robust Principal Component Analysis (RPCA, see \cite{CLMW11}), which is a general framework to extract a low-rank matrix and a sparse one from an incomplete and noisy observation. Mathematically, the most general form of the RPCA problem can be expressed as follows:
\begin{equation}\label{eq:rpca-lin}
\min_{X\in\RR^{m\times n},Y\in\RR^{m\times n}}\;\left\{~ \text{rank}(X) + \tau \|Y\|_0\;\; | \;\; \|\mathcal{P}_\Omega(X+Y-B)\|_F\leq \delta~\right\},
\end{equation}
where $\text{rank}(X)$ is the rank function, $\|Y\|_0$ represents the number of nonzero components of $Y$, $B$ is the observed data matrix, $\delta$ represents the standard deviation of Gaussian white noise, $\Omega$ is a subset of the index set of entries $\{1,2,\ldots,m\}\times \{1,2,\ldots,n\}$, $\mathcal{P}_\Omega(\cdot):\RR^{m\times n}\to \RR^{m\times n}$ is the orthogonal projection onto the span of matrices vanishing outside of $\Omega$ so that the $ij$-th entry of $\mathcal{P}_\Omega(B)$ is $B_{ij}$ if $(i,j)\in \Omega$ and zero otherwise. Such a problem has been proven to be NP-hard. Therefore, one of the most popular ways is the convex relaxation approach (see \cite{CLMW11}), which employs $\|X\|_*$ and $\|Y\|_1$ to approximate $\text{rank}(X)$ and $\|Y\|_0$, respectively. In \cite{SXY13}, the authors introduced a novel nonconvex formulation for the RPCA problem using the capped trace norm and the capped $\ell_1$-norm. Here, we accordingly follow their idea to reformulate \eqref{eq:rpca-lin} as follows:
\begin{equation}\label{eq:rPCA}
\min_{X,Y} \left\{\underbrace{\lno X \rno_* - g_1\left( X;\kappa_1\right)}_{\theta_1(X):=\lno X \rno_{*,\kappa_1}}  + \underbrace{\tau \left(\lno Y\rno_1 - g_2\left( Y;\kappa_2  \right) \right)}_{\theta_2(Y)=\lno Y \rno_{1,\kappa_2}} + \underbrace{ \frac{1}{2\lambda}\lno \mathcal{P}_\Omega(X+Y-B)\rno_F^2}_{h^+(X,Y)}\right\},
\end{equation}
where 
\begin{equation*}
g_1\left( X;\kappa_1\right) = \sum_i^{\min\{m,n\}} \max\left\{ \sigma_i(X)-\kappa_1,0\right\}\quad \text{and}\quad 
g_2\left( Y;\kappa_2  \right) = \sum_i^m\sum_j^n \max\left\{ |Y_{ij}|-\kappa_2,0 \right\},
\end{equation*}
and $\lno X \rno_{*,\kappa_1}$ and $\lno Y \rno_{1,\kappa_2}$ are the so-called capped trace norm and the capped $\ell_1$-norm, respectively.
Obviously, \eqref{eq:rPCA} falls into the case of model \eqref{eq:problem} with two separable DC parts, and our UBAMA is applicable to such a problem. According to \cite{SXY13}, we note that the subdifferential of $g_1\left( X;\kappa_1\right)$ is given by 
\begin{equation}\label{eq:subg1}
\partial g_1(X;\kappa_1):=\left\{ U\text{diag}(w)V^\top\;|\;w\in W^*\right\} 
\end{equation}
where $U$ and $V$ are the left and right singular vectors of $X$, respectively, and 
$$W^*=\left\{w\in\RR^r\;\Big{|}\;  w_i\in
\begin{cases}
\{1\}, & \text{if } \sigma_i(X)>\kappa_1, \\
\{0\}, & \text{if } \sigma_i(X)<\kappa_1, \\
[0,1], & \text{otherwise}.
\end{cases}
\right\}$$
with $r$ being the rank of $X$. Besides, we also have the subdifferential of $g_2\left( Y;\kappa_2  \right)$ given by
\begin{equation}\label{eq:subg2}
\partial g_2\left( Y;\kappa_2  \right)=\left\{ V\in\RR^{m\times n}\; \Big{|}\;
V_{ij} \in \begin{cases}   
\{1\}, &\text{if } Y_{ij}>\kappa_2,\\
[0,1], &\text{if } Y_{ij}=\kappa_2,\\
\{0\}, &\text{if } |Y_{ij}|<\kappa_2,\\
[-1,0], &\text{if } Y_{ij}=-\kappa_2,\\
\{-1\}, &\text{if } Y_{ij}<-\kappa_2.
\end{cases}\quad\right\}.
\end{equation}
Due to the appearance of $\mathcal{P}_\Omega$, we accordingly take the kernel functions as $\psi_k(\cdot) = \frac{1}{2}\|\cdot\|_{M_k}^2$ and $\varphi_k(\cdot)= \frac{1}{2}\|\cdot\|^2_{N_k}$ for the Bregman proximal terms $\B_{\psi_k}$ and $\B_{\varphi_k}$, where $M_k$ and $N_k$ are specified as $M_k=\frac{1}{\lambda}\left( \mu_k I-\mathcal{P}_\Omega^\top \mathcal{P}_\Omega\right)$ and $N_k = \frac{1}{\lambda}\left( \nu_kI-\mathcal{P}_\Omega^\top \mathcal{P}_\Omega\right)$ with appropriate $\mu_k>0$ and $\nu_k>0$, respectively, such that $M_k$ and $N_k$ are positive definite. With the above preparations, the algorithmic details of applying our UBAMA to \eqref{eq:rPCA} are summarized in Algorithm \ref{alg:ubama-two}.

\begin{algorithm}[!htb]
	\caption{UBAMA for~\eqref{eq:rPCA}.}\label{alg:ubama-two}
	\begin{algorithmic}[1]
		\STATE Select $M_k = \frac{1}{\lambda}\left( \mu_k I-\mathcal{P}_\Omega^\top \mathcal{P}_\Omega\right)$, $N_k = \frac{1}{\lambda}\left( \nu_kI-\mathcal{P}_\Omega^\top \mathcal{P}_\Omega\right)$ and starting points $X^0,Y^0$.
		\FOR{$k=0,1,2,\ldots$}
		\STATE Take $\xi^k \in \partial g_1(X^k;\kappa_1)$ by \eqref{eq:subg1}, and compute $X^{k+1}$ via
		\begin{align*}
		X^{k+1} & = \arg\min_X \left\{\lno X\rno_*  - \langle  X,\xi^{k+1} \rangle + \frac{1}{2\lambda}\lno \mathcal{P}_\Omega(X+Y^k-B)\rno_F^2 + \frac{1}{2}\lno X - X^k\rno_{M_k}^2\right\} \\
		& = \SVT\left( X^k- \frac{1}{\mu_k}\mathcal{P}_\Omega^\top \mathcal{P}_\Omega\left( X^k+Y^k-B\right)+\frac{\lambda}{\mu_k}\xi^{k+1}, \frac{\lambda}{\mu_k}\right).
		\end{align*}
		\STATE Take $\eta^{k+1}\in \partial g_2(Y^k;\kappa_2)$ by \eqref{eq:subg2}, and update $Y^{k+1}$ via
		\begin{align*}
		Y^{k+1} & = \arg\min_Y \left\{\tau\lno Y \rno_1 - \tau \langle Y,\eta^{k+1}\rangle  + \frac{1}{2\lambda}\lno \mathcal{P}_\Omega(X^{k+1}+Y-B)\rno_F^2 + \frac{1}{2}\lno Y-Y^k\rno_{N_k}^2\right\}\\
		& = \shrink \left( Y^k - \frac{1}{\nu_k}\mathcal{P}_\Omega^\top \mathcal{P}_\Omega\left( X^{k+1}+Y^k-B\right)+\frac{\lambda\tau}{\nu_k}\eta^{k+1} ,\frac{\lambda\tau}{\nu_k} \right).
		\end{align*}
		\ENDFOR
	\end{algorithmic}
\end{algorithm}

When both DC parts (i.e., $\theta_1(X)$ and $\theta_2(Y)$) are regarded as general nonconvex functions, model \eqref{eq:rPCA} can be solved via the aforementioned algorithms in Section \ref{Sec:Intro}, e.g., \eqref{eq:pAM} and \eqref{eq:PALM}. For the purpose of comparison, we employ the most recent alternating DC algorithm (denoted by ADCA) introduced in \cite{PHLH22} to solve \eqref{eq:rPCA}, where the underlying $X$-subproblem is solved approximately via the state-of-the-art ADMM (see \cite{GM75} and also \cite{BPCPE10,CHY12}) by setting the stopping tolerance as $10^{-4}$. As shown in Remark \ref{rem:sc2}, the ADCA algorithm can be viewed as one of special cases of our UBAMA. Note that the coupling term $h^+(X,Y)$ in \eqref{eq:rPCA} is indeed a convex function. So, when reformulating \eqref{eq:rPCA} as
\begin{equation}\label{eq:rPCAv}
\min_{Z:=(X,Y)} \left\{\underbrace{\lno X \rno_* +\tau \lno Y\rno_1 }_{f(Z)} -\underbrace{ \left(g_1\left( X;\kappa_1\right) +g_2\left( Y;\kappa_2  \right) \right)}_{g(Z)} + \underbrace{ \frac{1}{2\lambda}\lno \mathcal{P}_\Omega(X+Y-B)\rno_F^2}_{h(Z)}\right\},
\end{equation}
we can see that \eqref{eq:rPCAv} is precisely the same as the generalized DC programming discussed in the literature (e.g., \cite{CHZ22,TFT22,WCP18}). Therefore, we can employ the DCA-type methods to find solutions of \eqref{eq:rPCAv} (see Algorithm \ref{alg:DCA-ADMM} for details and denote it by DCA-ADMM), where the subproblem \eqref{eq:dca-admm-sub} is also solved by the ADMM with stopping tolerance $10^{-4}$. In accordance with the first-order optimality condition of \eqref{eq:rPCA}, we employ the following stopping criterion 
\begin{equation}\label{eq:StopTol}
{\rm Tol}:= \max\left\{ \frac{\lno X^{k+1} - X^k \rno_F}{\max\{1,\lno X^k \rno_F\}}, \frac{\lno Y^{k+1} - Y^k \rno_F}{\max\{1,\lno Y^k \rno_F\}} \right\} \leq \epsilon.
\end{equation}
to return a pair of approximate solutions for all algorithms.

\begin{algorithm}[!htb]
	\caption{DCA with ADMM solver for \eqref{eq:rPCA}.}	\label{alg:DCA-ADMM}
	\begin{algorithmic}[1]
		\STATE Select starting points $X^0$ and $Y^0$.
		\FOR{$k=0,1,2,\cdots$}
		\STATE Compute $\xi^{k+1} \in \partial g_1(X^k;\kappa_1)$ and $\eta^{k+1} \in \partial g_2(Y^k;\kappa_2)$ via \eqref{eq:subg1} and \eqref{eq:subg2}, respectively.
		\STATE Update $(X^{k+1},Y^{k+1})$ via solving the following minimization problem by ADMM:
		\begin{equation}\label{eq:dca-admm-sub}
		\min\left\{ \lno X \rno_*+ \tau \lno Y \rno_1 - \langle \xi^{k+1}, L \rangle  - \tau \langle \eta^{k+1}, Y\rangle + \frac{1}{2\mu}\lno \mathcal{P}_\Omega(X+Y-B)\rno_F^2\right\}.
		\end{equation}
		\ENDFOR
	\end{algorithmic}
\end{algorithm}

In this subsection, we consider \eqref{eq:rPCA} with two different datasets. One is a synthetic dataset generated in a random way. Another one is a real-world surveillance video dataset that has been widely tested in the literature. For the experiments on synthetic data, we first generate a rank-$r$ matrix $X^*$ as a product of $QR^\top$, where $Q$ and $R$ are independent $n\times r$ matrices whose elements are i.i.d. random variables sampled from the standard Gaussian distribution $\mathcal{N}(0,1)$. Then, we generate a sparse matrix $Y^*$ with only $5\%$ nonzero entries, which are uniformly distributed in the interval $[-10, 10]$. The ground truth matrix is then corrupted by adding white noise $\mathcal{E} \in \mathcal{N}(0,0.01)$. To form an incomplete observation, we further randomly generate a sample operator $S$ with sample rate ${\sf sr}$. Then, the observed incomplete matrix $B$ is generated by $S(X^*+Y^* + \mathcal{E})$. In our experiments, we set the regularization coefficients $\tau = \frac{1}{\sqrt{n}}$ and $\lambda = \sqrt{{\sf sr} \times \sqrt{8n\times {\sf sr}}\delta}$, where $\delta=0.01$ is the noise level and ${\sf sr}=0.9$. The stopping tolerance $\epsilon$ in \eqref{eq:StopTol} is specified as $\epsilon = 10^{-4}$. For the parameters of our UBAMA, we take $\mu_k=1.01$ and $\nu_k=1.01$. Since the data is generated in a random way, we report the averaged performance on $10$ trials in Table \ref{tab:rpca_synthetic}, which clearly shows that our UBAMA takes much less computing time than the other two DCA-type solvers for relatively large-scale problems.

\begin{table}[!tb]
	\caption{Numerical comparisons for RPCA with synthetic data.}\label{tab:rpca_synthetic}
	\centering
	\begin{tabular*}{\textwidth}{@{\extracolsep{\fill}}ccccccc}\toprule
		$(n,r)$ & Method & $\frac{\lno X-X^*\rno_F}{\lno X^*\rno_F}$ ($\times 10^{-3}$) & $\frac{\lno Y-Y^* \rno_F}{\lno Y^* \rno_F}$  &  Iter. & Time(s) \\ \midrule 
		\multirow{3}{*}{$(256,8)$}
		& UBAMA      & $4.0409$ & $0.32620$ & $16$ & $0.4398$ \\ 
		& ADCA       & $3.9027$ & $0.32615$ & $13$ & $5.8429$ \\ 
		& DCA-ADMM   & $4.3047$ & $0.32643$ & $5$ & $3.8280$ \\ \midrule
		\multirow{3}{*}{$(512,16)$}
		& UBAMA      & $2.9431$ & $0.32322$ & $20$ & $3.2122$ \\ 
		& ADCA       & $2.9880$ & $0.32325$ & $12$ & $30.0875$ \\ 
		& DCA-ADMM   & $3.3176$ & $0.32361$ & $6$ & $24.4760$ \\ \midrule
		\multirow{3}{*}{$(768,24)$}
		& UBAMA      & $2.2068$ & $0.31973$ & $23$ & $12.5701$ \\ 
		& ADCA       & $2.2633$ & $0.31979$ & $13$ & $101.4250$ \\ 
		& DCA-ADMM   & $2.4971$ & $0.32008$ & $6$ & $72.9432$ \\ \midrule
		\multirow{3}{*}{$(1024,32)$}
		& UBAMA      & $1.9992$ & $0.31508$ & $25$ & $39.9955$ \\ 
		& ADCA       & $2.0360$ & $0.31513$ & $13$ & $285.7886$ \\ 
		& DCA-ADMM   & $2.2624$ & $0.31548$ & $7$ & $248.3535$ \\ 
		\bottomrule
	\end{tabular*}
\end{table}
Now, we consider two real-world video datasets (i.e., {\sf Lobby} of size $(m,n)=(20480 \times 150)$ and {\sf Hallairport} of size $(m,n)=(25344 \times 100)$). In this experiment, we follow the way used for synthetic data to degrade the video (i.e., the same noise level and sampling way), and set $\tau = \frac{1}{\sqrt{\max(m,n)}}$,  $\lambda = (\sqrt{m}+\sqrt{n})\cdot\delta\cdot\sqrt{{\sf sr}}$, and $\epsilon=10^{-3}$. Besides, we take $\mu_k=1.01$ and $\nu_k=1.01$ for UBAMA. Here, we fix the noise level $\delta=0.01$ and conduct three scenarios on the sample rate, i.e., ${\sf sr}=\{0.70,0.80,0.90\}$. In Tables~\ref{tab:rpca_real-Lobby} and \ref{tab:rpca_real-Hallairpot}, we report the rank ($\text{rank}(X)$) of the low-rank part $X$, the number of nonzero components of $Y$ ($\|Y\|_0$), the objective function value of \eqref{eq:rPCA} (Obj.), the number of  iterations (Iter.), the computing time in seconds (Time(s)), the relative error (Err.) defined by 
$$\text{Err.}=\frac{\lno \mathcal{P}_\Omega(X+Y-A)\rno_F}{\lno \mathcal{P}_\Omega(A)\rno_F},$$
where $A$ is the original clean video. We see from Tables \ref{tab:rpca_real-Lobby} and \ref{tab:rpca_real-Hallairpot} that the three algorithms achieve the same rank and almost the same objective values and relative errors. Comparatively, both ADCA and DCA-ADMM perform a little better than our UBAMA in terms of the sparsity of $Y$. However, our UBAMA runs much faster than ADCA and DCA-ADMM. Moreover, it can be seen from the recovered frames in Figure \ref{fig:rpca_compare_images} that the three algorithms can achieve almost the same quality of the background, foreground and the recovered frames. Therefore, these computational results efficiently supports the idea of this paper.

\begin{table}[!tb]
	\caption{Numerical comparisons for RPCA with a real-world dataset: video {\sf Lobby}.}\label{tab:rpca_real-Lobby}
	\centering
	\begin{tabular*}{\textwidth}{@{\extracolsep{\fill}}cccccccc}\toprule
		{\sf sr}&Method & $\text{rank}(X)$ & $\lno Y \rno_0$ ($\times 10^6$) & Obj.  & Err.& Iter. & Time(s) \\ \midrule 
		& UBAMA & $4$ & $1.0191$ & $317.03$ & $0.0212$ & $93$ & $78.92$ \\ 
		0.70 & ADCA & $4$ & $1.0189$ & $317.03$ & $0.0212$ & $29$ & $653.17$ \\ 
		& DCA-ADMM & $4$ & $1.0189$ & $317.02$ & $0.0212$ & $4$ & $149.30$ \\ 
		\midrule
		& UBAMA & $5$ & $1.0894$ & $331.60$ & $0.0209$ & $60$ & $46.59$ \\ 
		0.80 & ADCA & $5$ & $1.0892$ & $331.59$ & $0.0209$ & $22$ & $424.97$ \\ 
		& DCA-ADMM & $5$ & $1.0892$ & $331.59$ & $0.0209$ & $4$ & $123.74$ \\ 
		\midrule
		& UBAMA & $6$ & $1.1466$ & $345.33$ & $0.0208$ & $37$ & $27.63$ \\ 
		0.90 & ADCA & $6$ & $1.1464$ & $345.33$ & $0.0208$ & $21$ & $395.49$ \\ 
		& DCA-ADMM & $6$ & $1.1465$ & $345.32$ & $0.0208$ & $4$ & $122.57$ \\
		\bottomrule
	\end{tabular*}
\end{table}

\begin{table}[!tb]
	\caption{Numerical comparisons for RPCA with real-world data: video {\sf Hallairport}.}\label{tab:rpca_real-Hallairpot}
	\centering
	\begin{tabular*}{\textwidth}{@{\extracolsep{\fill}}cccccccc}\toprule
		{\sf sr}&Method & $\text{rank}(X)$ & $\lno S \rno_0$ ($\times 10^6$) & Obj.  & Err. & Iter. & Time(s) \\ \midrule 
		& UBAMA & $9$ & $0.9714$ & $513.97$ & $0.0138$ & $162$ & $82.75$ \\ 
		0.70 & ADCA & $9$ & $0.9704$ & $513.64$ & $0.0138$ & $80$ & $1071.70$ \\ 
		& DCA-ADMM & $9$ & $0.9703$ & $513.53$ & $0.0138$ & $6$ & $170.21$ \\ 
		\midrule
		& UBAMA & $12$ & $1.0358$ & $557.16$ & $0.0137$ & $131$ & $66.92$ \\ 
		0.80 & ADCA & $12$ & $1.0354$ & $557.02$ & $0.0137$ & $85$ & $1140.02$ \\ 
		& DCA-ADMM & $12$ & $1.0355$ & $556.95$ & $0.0137$ & $6$ & $177.44$ \\ 
		\midrule
		& UBAMA & $15$ & $1.0889$ & $599.14$ & $0.0136$ & $113$ & $57.30$ \\ 
		0.90 & ADCA & $15$ & $1.0889$ & $599.09$ & $0.0136$ & $91$ & $1193.42$ \\ 
		& DCA-ADMM & $15$ & $1.0889$ & $599.03$ & $0.0136$ & $5$ & $149.77$ \\ 
		\bottomrule
	\end{tabular*}
\end{table}

\begin{figure}[!tb]
	\begin{center}
		\includegraphics[width=\textwidth]{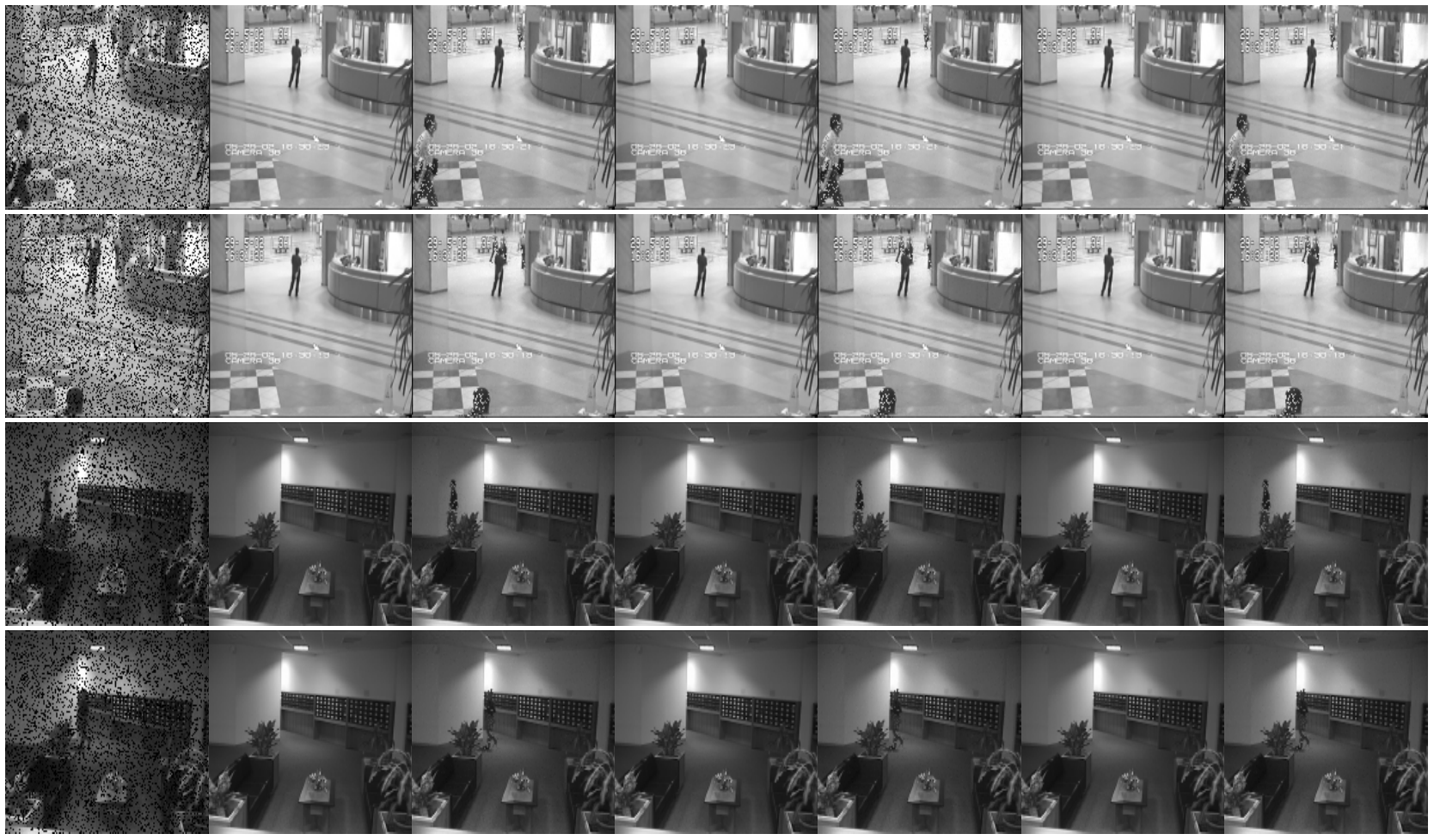}
	\end{center}
	\caption{Selected frames of the videos {\sf Hallairport} and {\sf Lobby}. The first column corresponds to the frames of the videos with $80\%$ observed information. The second and third columns are the recovered frames obtained by our UBAMA. The fourth and fifth columns are the frames obtained by ADCA. The last two columns are the frames obtained by DCA-ADMM. The second, fourth, and sixth columns are the background of the videos separated by UBAMA, ADCA, and DCA-ADMM. respectively.}
	\label{fig:rpca_compare_images}
\end{figure}

\subsection{Blind image deconvolution}\label{Sec:bid}
It is not difficult to observe that the nonconvexity of the above two models \eqref{eq:super-resolution} and \eqref{eq:rPCA} comes from their DC parts, while their coupling functions are indeed convex. In this subsection, we consider the well-studied yet challenging Blind Image Deconvolution (BID) problem (e.g., see \cite{CVR14,HNO06}), which has a nonconvex coupling function. Concretely, the task of BID is to recover both a sharp image $x\in\RR^m$ ($m=m_1\times m_2$ is the number of pixels of an image) and the unknown point spread function $y\in \RR^n$ (a small 2D blur kernel of size $n=n_1\times n_2$ pixels) from a given blurry and possibly noisy image $b$, i.e., solving the following inverse problem:
\begin{equation}
b = x \star y + \varepsilon,
\end{equation}
where $\varepsilon$ represents the noise and the operation $x \star y$ denotes the usual 2D modulo-$m_1,m_2$ discrete circular convolution operation defined by (and interpreting the image $x$ and the blur kernel $y$ as 2D arrays)
$$(x\star y)_{ij}=\sum_{s=0}^{n_1}\sum_{t=0}^{n_2} y_{st}x_{(i-s)_{{\text{mod} (m_1)}},(j-t)_{\text{mod} (m_2)}},\qquad 1\leq i\leq m_1,\;1\leq j\leq m_2,$$
which can also be expressed as the 2D discrete convolution by the matrix vector products of the form
$$u=x\star y \;\; \Leftrightarrow\;\; u=K(y)x \;\; \Leftrightarrow\;\; u=K(x)y$$
with $K(y)\in\RR^{m\times m}$ being a sparse matrix (each row holds the values of the blur kernel $y$) and $K(x)\in\RR^{m\times n}$ being a dense matrix (each column is given by a circularly shifted version of the image $x$). Generally, the blur kernel is assumed to be normalized, i.e., $y\in\Delta^n$, where $\Delta^n$ is the standard unit simplex given by \eqref{eq:simplex}. Moreover, the pixel intensities of the unknown sharp image $x$ are assumed to be normalized to the interval $[0,1]$, i.e., $x\in \mathbb{B}^m:=\{ x\in\RR^m\;|\; x_i\in [0,1],\; i=1,2\ldots,m\}$.
In what follows, we consider a classical BID model (e.g., see \cite{PDF15,PS16}), which reads as
\begin{equation}\label{eq:BID}
\min_{x,y} \;  \left\{\sum_{p=1}^8 \phi\left(\nabla_p x;\tau\right) + \frac{\lambda}{2}\left\| x \star y - b \right\|^2\; \Big{|}\; x\in \mathbb{B}^m,\; y\in\Delta^n\right\} ,
\end{equation}
where the function $\phi(\cdot~;\tau)$ is a differentiable robust error function promoting sparsity in its argument, which is  defined by
\begin{equation*}
\phi(x;\tau) = \sum_{i=1}^m\log\left(1+\tau x_i^2\right),\quad \text{for}\;\;x\in \RR^m,\;\; \tau>0,
\end{equation*}
and the linear operators $\nabla_p$ ($p=1,2,\ldots,8$) are the finite differences approximation to directional image gradients, which are given in \cite{PS16} as follows:
\begin{equation*}
\begin{array}{lll}
(\nabla_1 x)_{i,j}=x_{i+1,j}-x_{i,j}, & \qquad (\nabla_2 x)_{i,j}=x_{i,j+1}-x_{i,j}, & \\
(\nabla_3 x)_{i,j}=\frac{x_{i+1,j+1}-x_{i,j}}{\sqrt{2}}, & \qquad (\nabla_4 x)_{i,j}=\frac{x_{i+1,j-1}-x_{i,j}}{\sqrt{2}},&  \quad 1\leq i\leq m_1,\\
(\nabla_5 x)_{i,j}=\frac{x_{i+2,j+1}-x_{i,j}}{\sqrt{5}}, & \qquad (\nabla_6 x)_{i,j}=\frac{x_{i+2,j-1}-x_{i,j}}{\sqrt{5}}, &\quad 1\le j\le m_2.\\
(\nabla_7 x)_{i,j}=\frac{x_{i+1,j+2}-x_{i,j}}{\sqrt{5}}, & \qquad (\nabla_8 x)_{i,j}=\frac{x_{i-1,j+2}-x_{i,j}}{\sqrt{5}},&
\end{array}
\end{equation*}
In our experiments, we assume natural boundary conditions, i.e., $(\nabla_p)_{i,j}=0$, whenever the operator references a pixel location that lies outside the domain. However, it has been emphasized in \cite{PS16} that model \eqref{eq:BID} performs well for favor sharp images, while it could possibly be a bad choice for textured images. 

By invoking the indicator functions associated to $\mathbb{B}^m$ and $\Delta^n$, we can easily reformulate \eqref{eq:BID} as 
\begin{equation}\label{eq:BIDr}
\min_{x,y} \;  \left\{\;\underbrace{\mathcal{I}_{\mathbb{B}^m}(x)}_{f_1(x)}\;+\;\underbrace{\mathcal{I}_{\Delta^n}(y)}_{f_2(y)}\;-\;\underbrace{\left(-\sum_{p=1}^8 \phi\left(\nabla_p x;\tau\right) - \frac{\lambda}{2}\left\| x \star y - b \right\|^2\right)}_{h^-(x,y)} \;\right\} ,
\end{equation}
which is clearly a special case of model \eqref{eq:problem}. Moreover, it has been shown in \cite{PS16} that the coupling function $h^-(x,y)$ defined in \eqref{eq:BIDr} is smooth with block Lipschitz continuous gradients given by
\begin{equation*}
\begin{cases}
\nabla_x h^-(x,y)=-2\tau \sum_{p=1}^{8}\nabla_p^\top \text{\bf vec}\left(\frac{(\nabla_p x)_{i,j}}{1+\tau(\nabla_p x)_{i,j}^2}\right)_{i,j=1}^{m_1,m_2}-\lambda K^\top(y)\left(K(y)x-b\right), \\
\nabla_y h^-(x,y) =-\lambda K^\top(x)\left(K(x)y-b\right),
\end{cases}
\end{equation*}
where the operation ``$\text{\bf vec}(\cdot)$'' corresponds to the formation of a vector from the values passed to its argument. In \cite{PS16}, we notice that their method requires a solver to compute the projection onto the unit simplex with $\mathcal{O}(n\log n)$ time complexity. In this section, we shall consider two different Bregman proximal terms for the $y$-subproblem for the purpose of highlighting the flexibility of our algorithmic framework, especially of showing that one of our choices can avoid computing the projection onto the unit simplex.  Below, we elaborate the details of applying UBAMA to \eqref{eq:BID} as follows.
\begin{itemize}
\item Firstly, the $x$-subproblem is specified as
\begin{equation}\label{eq:bid-x}
x^{k+1} = \arg\min_{x\in\mathbb{B}^m} \left\{\mathcal{I}_{\mathbb{B}^m}(x) - \left\langle x-x^k,  \nabla_x h^-(x^k,y^k) \right\rangle + \B_{\psi_k}(x,x^k)  \right\}.
\end{equation}
Considering the simple structure of $\mathbb{B}^m$, we take $\psi_k(x)=\frac{c_k}{2}\|x\|^2$ in \eqref{eq:bid-x}. Then, the updating scheme of $x$ immediately reads as 
\begin{align}\label{eq:bid-xn}
x^{k+1} &= {\Proj}_{\mathbb{B}^m}\left(x^k + \frac{1}{c_k} \nabla_x h^-(x^k,y^k)  \right) \nonumber \\
&=\max\left\{\min\left\{x^k +\frac{1}{c_k} \nabla_x h^-(x^k,y^k),1\right\},0\right\}.
\end{align}
\item Secondly, the $y$-subproblem is specified as
\begin{equation}\label{eq:bid-y}
y^{k+1}= \arg\min_{y\in\Delta^n} \left\{\mathcal{I}_{\Delta^n}(y) - \left\langle y-y^k, \nabla_y h^-(x^{k+1},y^k)\right\rangle + \B_{\varphi_k}(y,y^k)\right\}.
\end{equation}
For the above subproblem, we have two options on $\varphi_k(\cdot)$.
\begin{itemize}
\item When taking $\varphi_k(y)=\frac{d_k}{2}\|y\|^2$, the iterative scheme \eqref{eq:bid-y} immediately is specified as
\begin{equation}\label{eq:bid-y-e}
y^{k+1}=\Proj_{\Delta^n}\left(y^k + \frac{1}{d_k}\nabla_y h^-(x^{k+1},y^k) \right), 
\end{equation}
which can be solved by the method introduced in \cite{DSSC08}.
\item When taking $\varphi_k(y)=d_k \sum_{i=1}^n y_i\log(y_i)$, the iterative scheme \eqref{eq:bid-y} reads as
\begin{equation}\label{eq:bid-y-ne}
y^{k+1}=\frac{y^k\odot e^{-\varsigma^k/d_k}}{\sum_{j=1}^m y_j^ke^{-\varsigma^k_j/d_k}},
\end{equation}
where $\varsigma^k:=\nabla_y h^-(x^{k+1},y^k)$ and ``$\odot$'' is the component-wise product of vectors.
\end{itemize}
Clearly, we can see from \eqref{eq:bid-y-e} and \eqref{eq:bid-y-ne} that the latter has an explicit form, which is simpler than the former. So, \eqref{eq:bid-y-ne} can save some computational cost to accelerate our algorithm.
\end{itemize}

Now, we turn our attention to the numerical behaviors of the UBAMA equipped with different Bregman kernel functions. Here, we consider two blurry images {\sf Peppers} ($384\times 512$) and {\sf Books} ($512\times 340$), which are convoluted by blur kernels of size $27 \times 13$ and $31\times 31$, respectively. Empirically, we choose the model parameters as $\lambda = 5\times 10^5$ and $\tau= 10^4$ for this experiment. The UBAMA starts from $x^0$ and $y^0$, which are initialized to be the blurry image $b$ and  the average filters of size $n$, respectively. We set the maximal iterate number as $2000$ and the algorithmic parameters $c_k$ and $d_k$ are determined by the backtracking line search. In Figure \ref{fig:BID-images}, we present the original images, blurred images, and the recovered images by our UBAMA equipped with two different Bregman kernel functions. It is easy to see that both UBAMA variants could recover blur kernels and piecewise smooth images. Note that the BID problem is usually extremely ill-posed and highly nonconvex. In this situation, the accuracy of solving subproblems and the proximal parameters possibly affects the numerical behaviors of the UBAMA significantly. As an illustration, we see from Figure \ref{fig:BID_obj-SNR} that although the UBAMA equipped with two \eqref{eq:bid-y-e} and \eqref{eq:bid-y-ne} achieves almost the same objective values, \eqref{eq:bid-y-e} and \eqref{eq:bid-y-ne} have the opposite performance in terms of SNR values for {\sf Peppers} and {\sf Books}. However, when we are only concerned with the objective values, the UBAMA equipped with \eqref{eq:bid-y-ne} takes a little less computing time than the other one to achieve an acceptable objective value, which further supports that our UBAMA provides an efficient way to design customized algorithms for some real-world applications.

\begin{figure}[!tb]
	\begin{center}
		\includegraphics[width=0.24\textwidth]{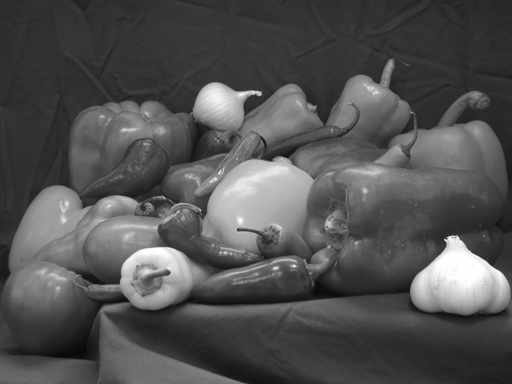}
		\includegraphics[width=0.24\textwidth]{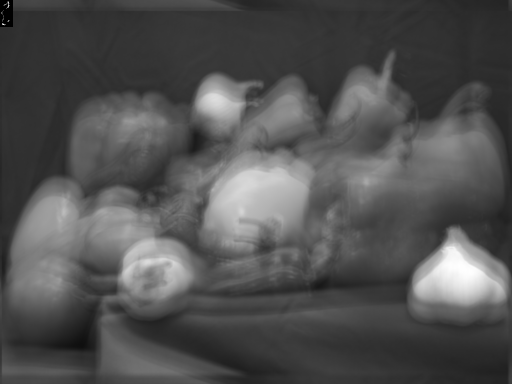}
		\includegraphics[width=0.24\textwidth]{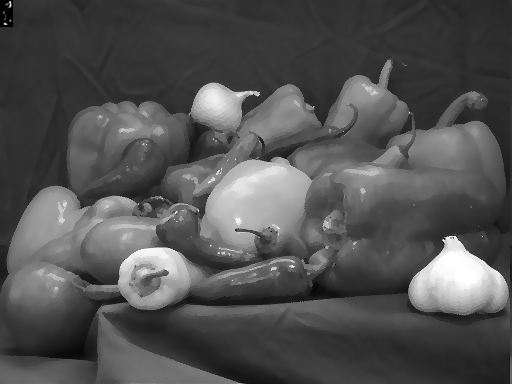}
		\includegraphics[width=0.24\textwidth]{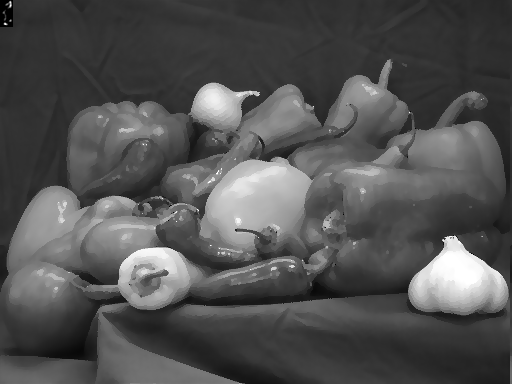}\\
		\includegraphics[width=0.24\textwidth]{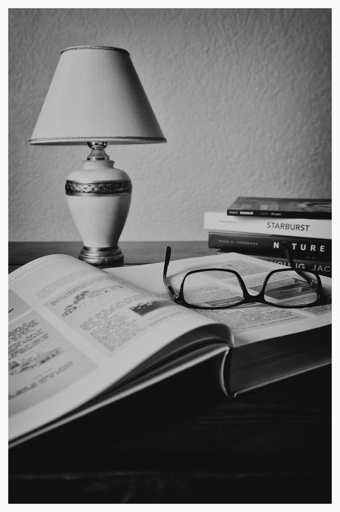}
		\includegraphics[width=0.24\textwidth]{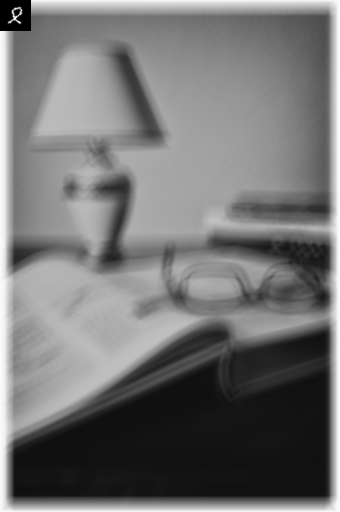}
		\includegraphics[width=0.24\textwidth]{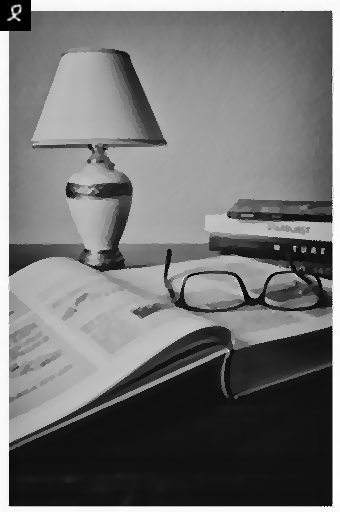}
		\includegraphics[width=0.24\textwidth]{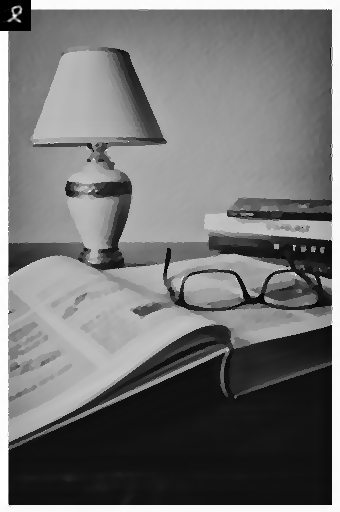}
	\end{center}
	\caption{BID restoration results for {\sf Peppers} and {\sf Books}. From left to right: clean image and blurred image, restored images by UBAMA equipped with $\varphi_k(y) = \sum_i y_i\log (y_i)$ and $\varphi_k(y) = 0.5 \|y\|^2$.}\label{fig:BID-images}
\end{figure}


\begin{figure}[!tb]
	\begin{center}
		\includegraphics[width = 0.45\textwidth]{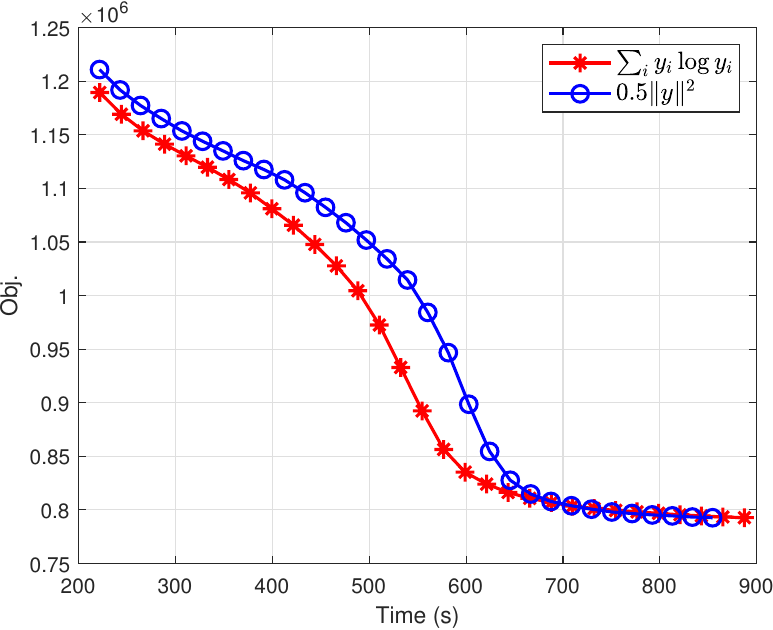}
		\includegraphics[width = 0.45\textwidth]{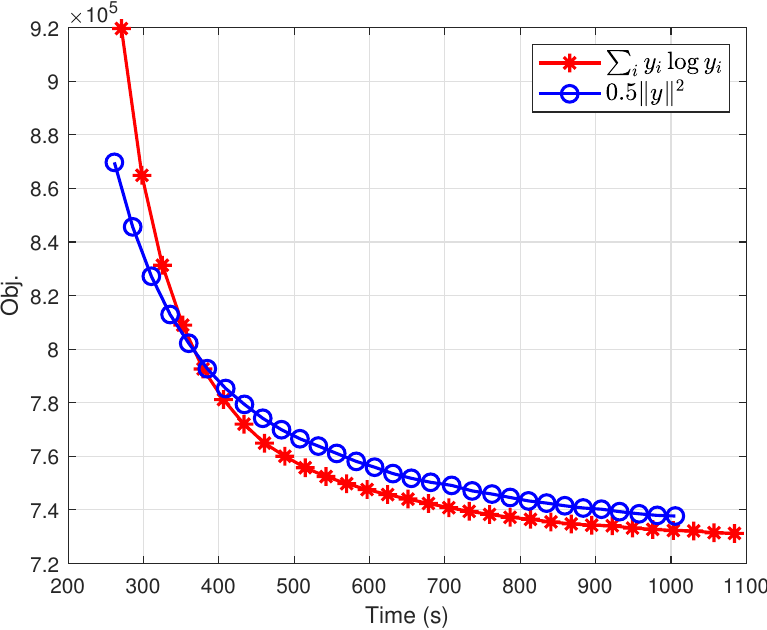}
		\includegraphics[width = 0.45\textwidth]{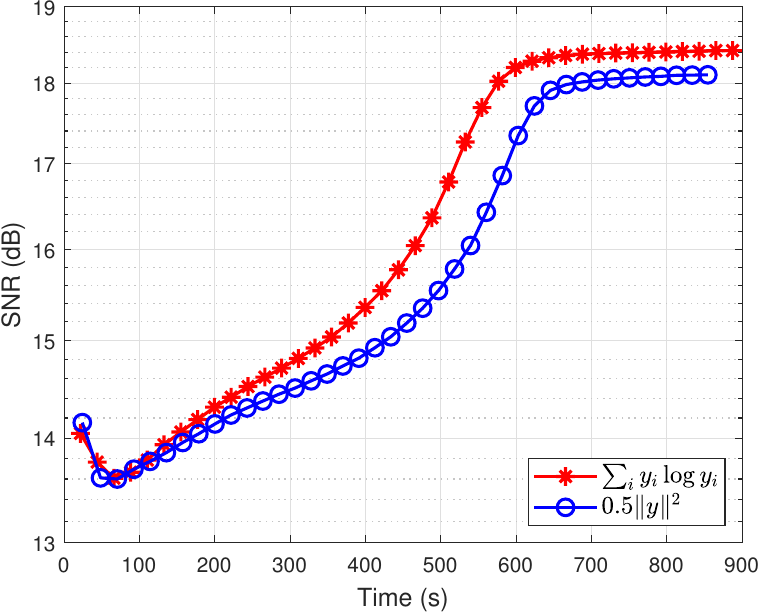}
		\includegraphics[width = 0.45\textwidth]{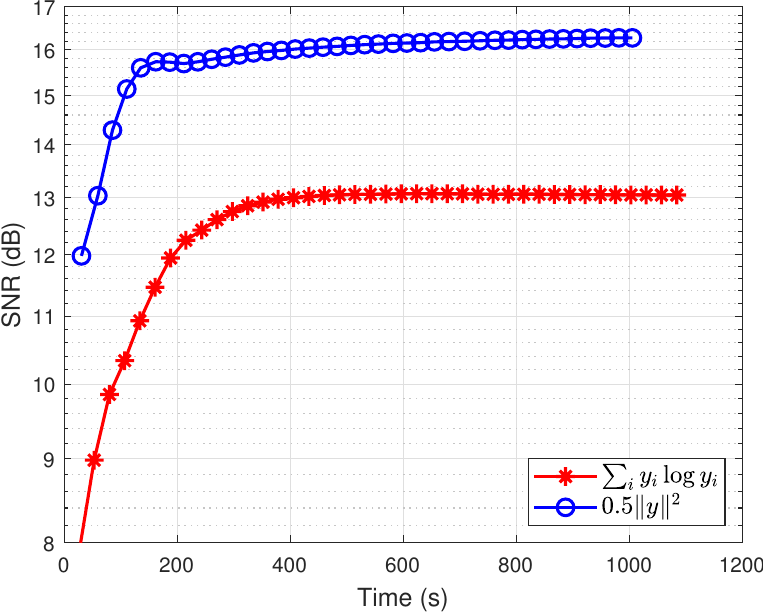}
	\end{center}
	\caption{Evolution of objective function and SNR values with respect to computing time, respectively. From left to right: {\sf Peppers} and {\sf Books}.}\label{fig:BID_obj-SNR}
\end{figure}

\section{Conclusions} \label{Sec:Con}
In this paper, we considered a class of structured nonconvex optimization problem, which is called generalized DC programming. The main computational challenges are the nonsmoothness, nonconvexity, and nonseparability appeared in the objective function. Designing an algorithm to circumvent the above difficulty is very important to efficiently solving the problem under consideration. So, we proposed a Unified Bregman Alternating Minimization Algorithm (UBAMA), which combines the novel spirits of DCA, alternating minimization, and Bregman regularization. Our UBAMA provides a flexible way to design customized algorithms for some real-world problems so that they often enjoy easy subproblems, which also is supported by some applications in image processing. Moreover, a series of numerical experiments demonstrated that our UBAMA performs well on solving generalized DC programming. In recent years, the inertial technique is widely used to accelerate many first-order optimization methods. So, we will consider some acceleration on our UBAMA by using the inertial technique in future. On the other hand, although our UBAMA enjoys easy subproblems for many sparse and low-rank optimization problems, there are some cases that we need to call solvers for finding solutions of subproblems. Therefore, designing inexact variants of UBAMA is also one of our future concerns. Finally, the stability of nonconvex optimization methods is very important for applications. So, we will also pay our attention on the stability of UBAMA in future.

\section*{Acknowledgments}
The authors would lik to thank Professors Thomas Pock and Shoham Sabach for sharing their code of \cite{PS16} with us.


\end{document}